\documentclass[11pt]{article}
\usepackage[a4paper,margin=1in]{geometry}
\usepackage[T1]{fontenc}
\usepackage[utf8]{inputenc}
\usepackage{lmodern}
\usepackage{microtype}
\usepackage{amsmath,amssymb,amsthm,mathtools}
\usepackage{enumitem}
\usepackage{hyperref}
\usepackage[nameinlink,noabbrev]{cleveref}
\usepackage[
  backend=biber,
  url=false,
  giveninits=true
]{biblatex}

\DeclareFieldFormat[article]{title}{#1}
\DeclareFieldFormat[online]{title}{#1}
\DeclareFieldFormat{eprint:arxiv}{arXiv\addcolon\space\mbox{\nolinkurl{#1}}}
\DeclareFieldFormat[inproceedings]{title}{#1}
\DeclareFieldFormat[incollection]{title}{#1}
\DeclareFieldFormat[inbook]{title}{#1}
\renewbibmacro{in:}{}
\DeclareNameAlias{author}{family-given}
\DeclareNameAlias{editor}{family-given}
\DeclareNameAlias{translator}{family-given}
\DeclareNameAlias{sortname}{family-given}
\DeclareNameAlias{default}{family-given}
\AtEveryBibitem{%
  \clearfield{issn}%
  \clearfield{url}%
  \clearfield{urlyear}%
  \clearfield{urlmonth}%
  \clearfield{urlday}%
}
\hypersetup{
  pdftitle={The Decomposition Theorem for Canonical Processes with Log-Concave Tails},
  colorlinks=true,
  linkcolor=blue,
  citecolor=blue,
  urlcolor=blue
}

\newtheorem{theorem}{Theorem}[section]
\newtheorem{lemma}[theorem]{Lemma}
\newtheorem{proposition}[theorem]{Proposition}
\newtheorem{corollary}[theorem]{Corollary}
\theoremstyle{definition}

\theoremstyle{remark}

\newcommand{\R}{\mathbb R}
\newcommand{\N}{\mathbb N}
\newcommand{\Z}{\mathbb Z}
\newcommand{\E}{\textsf E}
\renewcommand{\P}{\textsf P}
\newcommand{\ind}{\mathbf 1}
\newcommand{\cA}{\mathcal A}
\newcommand{\gexp}{\gamma_{\mathrm{exp}}}
\newcommand{\conv}{\operatorname{conv}}
\newcommand{\clconv}{\overline{\operatorname{conv}}}
\newcommand{\eqdist}{\stackrel{\mathrm d}=}
\numberwithin{equation}{section}
\allowdisplaybreaks[2]
\title{The Decomposition Theorem for Canonical Processes with Log-Concave Tails}
\author{
Xuanang Hu\thanks{Shandong University, Jinan, China. Email: xuananghu7@gmail.com}
\and
Hanchao Wang\thanks{Shandong University, Jinan, China. Email: wanghanchao@sdu.edu.cn}
\and
Xinglong Wu\thanks{Shandong University, Jinan, China. Email: wuxinglong3@gmail.com}
}
\date{}

\begin{document}
\maketitle

\begin{abstract}
We prove a decomposition theorem for canonical processes generated by
independent symmetric random variables with log-concave tails.
We characterize the expected supremum, up to universal constants,
both in terms of a decomposition of the index set and in terms of
admissible partitions. As a consequence we answer a question of
Lata\l a concerning convex hulls.


\end{abstract}

\section{Introduction and main results}\label{sec:intro}

We study the expected suprema of canonical processes. For a process
$(X_t)_{t\in T}$ with $T\ne\emptyset$, we define $\E\sup_{t\in T}X_t$
as the supremum of
$\E\sup_{t\in F}X_t$ over all finite nonempty subsets $F$ of $T$.

For a centered Gaussian process $(G_t)_{t\in T}$, put
$d(s,t)=(\E|G_s-G_t|^2)^{1/2}$. Dudley \cite{Dudley67} gave
an upper bound for $\E\sup_{t\in T}G_t$ in terms of the entropy
of $(T,d)$. Fernique \cite{Fernique75} obtained an upper bound
in terms of majorizing measures. Talagrand \cite{Talagrand87}
proved the converse. We recall the formulation of this result
in terms of partitions \cite{Talagrand96,TalagrandBook}.

An increasing sequence $(\cA_k)_{k\ge0}$ of partitions of $T$
is called admissible if $\cA_0=\{T\}$ and $|\cA_k|\le2^{2^k}$
for $k\ge1$. By increasing we mean that each member of $\cA_{k+1}$
is contained in a member of $\cA_k$. We denote by $\cA_k(t)$
the member of $\cA_k$ containing $t$. For $\alpha>0$, put
\[
 \gamma_\alpha(T,d)=\inf_{(\cA_k)}\sup_{t\in T}
 \sum_{k\ge0}2^{k/\alpha}\Delta(\cA_k(t),d),
\]
where the infimum is over all admissible sequences and $\Delta(A,d)=\sup_{s,t\in A}d(s,t)$.
The majorizing-measure theorem states that
$\E\sup_{t\in T}G_t\asymp\gamma_2(T,d)$.
Here and in the sequel $c,C>0$ denote universal constants
which may change from one occurrence to the next.
We write $a\asymp b$ if $a\le Cb$ and $b\le Ca$.
For a random variable $Z$, put $\|Z\|_p=(\E|Z|^p)^{1/p}$,
$1\le p<\infty$, and $\|Z\|_\infty=\operatorname{ess\,sup}|Z|$.

Let $F(A)=\E\sup_{t\in A}G_t$. Sudakov's minoration and Gaussian
concentration give the following estimate. Suppose $A_1,\ldots,A_m$
are nonempty subsets of $A$, with $d(s,t)\ge a$ for $s\in A_i$,
$t\in A_j$, $i\ne j$, and $\Delta(A_i,d)\le c_0a$ for all $i$.
For $c_0$ sufficiently small,
$F(A)\ge\min_{i\le m}F(A_i)+ca\sqrt{\log m}$.
In this estimate both the separation of the sets and their
diameters are measured by the same distance $d$.

For $T\subset\ell_2$, the canonical Gaussian process is
$G_t=\sum_i t_i g_i$, where the $g_i$ are independent standard
Gaussian variables. In this case $d=d_2$, where
$d_q(s,t)=\|s-t\|_q$ for $q=2,\infty$. Replacing the $g_i$
by independent symmetric signs $\varepsilon_i$, we obtain the
Bernoulli process $B_t=\sum_i t_i\varepsilon_i$. Put
$b(T)=\E\sup_{t\in T}B_t$.
There are two elementary upper bounds for $b(T)$:
$\sup_{t\in T}\|t\|_1$ and $C\gamma_2(T,d_2)$. Neither bound is
sharp in general. For $T=\{e_1,\ldots,e_m\}$, where the $e_i$
are the unit vectors, $b(T)\le1$ but
$\gamma_2(T,d_2)\asymp\sqrt{\log m}$ for $m\ge2$.
For $T=\{0,(1,\ldots,1)\}\subset\R^m$, the first bound is $m$
whereas $b(T)\asymp\sqrt m$.

For $T\subset T_1+T_2$ we have $b(T)\le b(T_1)+b(T_2)$.
Talagrand asked whether a decomposition could always be found
for which $\sup_{t\in T_1}\|t\|_1+\gamma_2(T_2,d_2)\le Cb(T)$.
This conjecture was proved by Bednorz and Lata\l a \cite{BL}:
\[
 b(T)\asymp\inf_{T\subset T_1+T_2}
 \left\{\sup_{t\in T_1}\|t\|_1+\gamma_2(T_2,d_2)\right\}.
\]

For Bernoulli processes, the Sudakov lower bound
$ca\sqrt{\log m}$ holds for $m$ points separated by $a$ in $d_2$
if their $d_\infty$-diameter is at most $a/\sqrt{\log m}$.
Talagrand \cite{Talagrand93,Talagrand94Bernoulli,Talagrand94Canonical}
introduced chopping maps to reduce the supremum norm of the
coefficients. Using the distances given by chopping maps of different lengths,
he proved a weaker form of the conjecture with an $\ell_p$ bound
in place of the $\ell_1$ bound, with a constant depending on $p>1$.
The difficulty is that chopping may considerably decrease the
$\ell_2$-distance.
An estimate for the expected supremum on subsets of small
diameter before chopping need not apply to subsets of small
diameter after chopping. Bednorz and Lata\l a \cite{BL} used
in addition a principle which allows one to remove some of the
variables introduced by the chopping maps.
See also \cite{BMMBernoulli,LiuZadik} for other proofs of the
Bernoulli theorem.

We consider the same problem for canonical processes with arbitrary
log-concave tails. Let $(Y_i)_{i\ge1}$ be independent symmetric
random variables and put
\begin{equation}\label{eq:intro-tail}
 N_i(u)=-\log\P(|Y_i|\ge u),\qquad
 q_i(p)=\sup\{u\ge0:N_i(u)\le p\},\qquad u,p\ge0.
\end{equation}
We assume that the $N_i$ are convex. The $q_i$ are then nonnegative,
nondecreasing and concave. Omitting variables which are identically zero, we have $q_i(1)>0$
and normalize so that $q_i(1)=1$. In finite dimension this amounts
to replacing $Y_i$ by $Y_i/q_i(1)$ and $t_i$ by $q_i(1)t_i$.
In what follows we work with the normalized variables.
Let $(E_i)$ be independent mean-one exponentials, independent of the
signs $(\varepsilon_i)$. We have
\begin{equation}\label{eq:intro-quantile}
 (Y_i)_i\eqdist(\varepsilon_iq_i(E_i))_i.
\end{equation}
Put $b_i=\lim_{p\to\infty}q_i(p)\in[1,\infty]$. Since
$q_i(p)\le\max\{1,p\}$, the second moments of the $Y_i$ are uniformly
bounded. For $t\in\ell_2$ and nonempty $T\subset\ell_2$ put
\begin{align}
 Y_t&=\langle t,Y\rangle=\sum_{i\ge1}t_iY_i,\label{eq:intro-process}\\
 S_Y(T)&=\E\sup_{t\in T}Y_t.\label{eq:intro-width}
\end{align}
The series in \eqref{eq:intro-process} converges in $L^2$ by
independence and symmetry. We identify $\R^n$ with the subspace of $\ell_2$ supported on
$\{1,\ldots,n\}$.

For variables with tails $\P(|Y_i|\ge u)=\exp(-u^a)$,
$1\le a<\infty$, a characterization of $S_Y(T)$ was obtained
by Talagrand \cite{Talagrand94Canonical}. Lata\l a
\cite{LatalaSudakov} extended this result to variables with
log-concave tails satisfying $N_i(2u)\le K N_i(u)$ for all $i$
and $u>0$, with a fixed $K$ (the $\Delta_2$ condition).
Put $d_p^Y(s,t)=\|Y_s-Y_t\|_p$ for $p\ge1$. Under $\Delta_2$,
Lata\l a's result gives the equivalence, with constants depending
only on $K$, between $S_Y(T)$ and
\[
 \inf_{(\cA_k)}\sup_{t\in T}
 \sum_{k\ge0}\Delta(\cA_k(t),d_{2^k}^Y).
\]
Lata\l a and Tkocz \cite{LT} extended this characterization
to independent symmetric variables satisfying
$\|Y_i\|_p\le\alpha(p/q)\|Y_i\|_q$ for $p\ge q\ge2$ and
$\|Y_i\|_{\lambda p}\ge2\|Y_i\|_p$ for $p\ge2$, with fixed
$\alpha\ge1$ and $\lambda>1$. The constants depend only on
$\alpha$ and $\lambda$. The first condition holds for variables
with log-concave tails. The second follows from $\Delta_2$.
In fact, the proof of the lower bound in \cite{LT} reduces the
problem to Lata\l a's theorem under $\Delta_2$.

The upper bound on the growth of moments alone does not suffice
for this characterization. Under $\Delta_2$ we have
$\sum_{k=0}^n\|Y_i\|_{2^k}\le C_K\|Y_i\|_{2^n}$, with $C_K$
depending only on $K$. Without $\Delta_2$ the infimum above
need not be bounded by a constant times $S_Y(T)$, as was noted
in \cite{LT} for Bernoulli variables. The following example
shows that this is not restricted to bounded variables. Suppose
$\P(|Y_i|\ge u)=\exp(1-2^u)$, $u\ge0$, for all $i$, and take
$T=\{e_1,\ldots,e_m\}$. Then $\|Y_i\|_p\asymp\log(p+1)$
and $S_Y(T)\le C\log\log m$ for large $m$. Put
$k_* =\min\{k:2^{2^k}\ge m\}$. For any admissible sequence, some member of $\cA_{k_*-1}$
contains two distinct points $s,t$. We have $s\in\cA_k(t)$ for
$k<k_*$. Since $\|Y_i-Y_j\|_p\ge\|Y_i\|_p$ for $i\ne j$,
$\sum_{k\ge0}\Delta(\cA_k(t),d_{2^k}^Y)\ge
c\sum_{k<k_*}(k+1)\ge c(\log\log m)^2$.
The upper bound for $S_Y(T)$ in terms of the infimum above
holds without $\Delta_2$ (see \cite{LT}). The example shows
that the reverse bound fails even for unbounded variables.
Lata\l a's Sudakov minoration \cite{LatalaSudakov}, on the
other hand, holds for arbitrary log-concave tails.

In \cite{HWW} we used information-theoretic methods to obtain
a majorizing-measure characterization without the $\Delta_2$
condition. See also Bednorz, Martynek and Meller \cite{BMM26}
for a variational proof. Here we characterize $S_Y(T)$ in
terms of admissible partitions and, as in the Bernoulli case,
use these partitions to obtain a decomposition of $T$.

We have two upper bounds for $S_Y(T)$.
One is $\E\sup_{t\in T}\sum_i|t_i|\,|Y_i|$. For the second,
put $\gexp(T)=\gamma_2(T,d_2)+\gamma_1(T,d_\infty)$.
Comparison with a symmetric exponential process gives
$S_Y(T)\le C\gexp(T)$ (see Section~\ref{sec:upper}). Thus for
$t_0\in T$ and $T-t_0\subset T_1+T_2$ we have
\[
 S_Y(T)\le\E\sup_{x\in T_1}\sum_i|x_i|\,|Y_i|+C\gexp(T_2).
\]
Our first result is the converse.

\begin{theorem}\label{thm:intro-decomp}
For every nonempty $T\subset\ell_2$,
\begin{equation}\label{eq:intro-decomp}
 S_Y(T)\asymp
 \inf
 \left\{
 \E\sup_{x\in T_1}\sum_i|x_i|\,|Y_i|
 +\gexp(T_2)
 \right\},
\end{equation}
where the infimum is over $t_0\in T$ and nonempty sets
$T_1,T_2\subset\ell_2$ such that
\[
 T-t_0\subset T_1+T_2.
\]
\end{theorem}

For Bernoulli variables the first term in \eqref{eq:intro-decomp}
is $\sup_{x\in T_1}\|x\|_1$. By the Bernoulli theorem the same
infimum with $\gamma_2(T_2,d_2)$ in place of $\gexp(T_2)$
is equivalent to $b(T)$.

To prove Theorem~\ref{thm:intro-decomp}, we construct
admissible partitions of $T$ using the functions $\Psi_i$
introduced in \cite{HWW}.

For each $i$, define $V_i$ on $[-1,1]\cup\{u\in\R:1<|u|\le b_i,
N_i(|u|)<\infty\}$ by
\begin{equation}\label{eq:intro-V}
 V_i(u)=
 \begin{cases}
 u^2,&|u|\le1,\\
 2N_i(|u|)-1,&|u|>1.
 \end{cases}
\end{equation}
If $b_i<\infty$, then $\pm b_i$ belong to the domain if and only if
$\P(|Y_i|=b_i)>0$. If $b_i>1$, then $N_i(1)=1$ and
$N'_{i,+}(1)\ge1$, so that $V_i$ is convex. For $0<u<b_i$, put
$\sigma_i(u)=V'_{i,+}(u)$ and, for $x\ge0$, put
\begin{equation}\label{eq:intro-Psi}
 \Psi_i(x)=\int_0^{b_i}\sigma_i(u)\ind_{\{\sigma_i(u)<x\}}\,du.
\end{equation}
We have
\begin{equation}\label{eq:intro-Psi-basic}
 \Psi_i(x)=x^2/4\quad(0\le x\le2),\qquad
 \Psi_i(x)\ge\min\{1,x^2/4\}.
\end{equation}
Put $r=8$. For $j\in\Z$ and $s,t\in\ell_2$, put
$\phi_j(s,t)=\sum_i\Psi_i(r^j|s_i-t_i|)$.
Since $|s_i-t_i|\le2\max\{|s_i-u_i|,|u_i-t_i|\}$, we have
$\phi_j(s,t)\le\phi_{j+1}(s,u)+\phi_{j+1}(u,t)$.

For nonempty $T\subset\ell_2$, define $\gamma_Y(T)$ as
the infimum of the numbers $M\ge0$ for which there exist an
admissible sequence $(\cA_n)$ of partitions of $T$ and
numbers $j_n(A)\in\Z\cup\{+\infty\}$, $A\in\cA_n$,
satisfying
\begin{gather}
 s,t\in A\in\cA_n
 \quad\Longrightarrow\quad
 \phi_{j_n(A)}(s,t)\le2^n,
 \label{eq:intro-gamma-diam}\\
 \Delta(T,d_2)\le r^{-j_0(T)},\notag\\
 j_n(\cA_n(t))\le j_{n+1}(\cA_{n+1}(t))
 \qquad(t\in T),\notag\\
 \sup_{t\in T}\sum_{n\ge0}2^n
 r^{-j_n(\cA_n(t))}\le M.
 \label{eq:intro-gamma}
\end{gather}
We allow $j_n(A)=+\infty$ only if $A=\{t\}$, and then put
$r^{-j_n(A)}=0$ and $\phi_\infty(t,t)=0$.
If no finite $M$ satisfies these conditions, set $\gamma_Y(T)=+\infty$.
Our second result is the following.

\begin{theorem}\label{thm:intro-main}
There is a universal constant $C$ such that, for every nonempty
$T\subset\ell_2$,
\begin{equation}\label{eq:intro-main}
 C^{-1}\gamma_Y(T)\le S_Y(T)\le C\gamma_Y(T).
\end{equation}
\end{theorem}

For Bernoulli variables $\Psi_i(x)=\min\{x^2/4,1\}$, so
$\phi_j(s,t)=\sum_i\min\{r^{2j}|s_i-t_i|^2/4,1\}$.
These are the functions in \cite[(10.167)]{TalagrandBook}
with $s,t$ replaced by $s/2,t/2$. Theorem~\ref{thm:intro-main}
extends the characterization in \cite[Section~10.14]{TalagrandBook}.

We prove the two theorems together. We first construct
partitions satisfying \eqref{eq:intro-gamma-diam}--\eqref{eq:intro-gamma}
with $M\le CS_Y(T)$ when $S_Y(T)<\infty$. Given any partitions
satisfying these conditions with $M<\infty$,
we construct maps $a,v:T\to\ell_2$ and a point $t_0\in T$
such that $t-t_0=a(t)+v(t)$ and
\[
 \E\sup_{t\in T}\sum_i|a_i(t)|\,|Y_i|+\gexp(v(T))\le CM.
\]
Taking $T_1=a(T)$ and $T_2=v(T)$ proves the reverse inequality
in \eqref{eq:intro-decomp} and the bound
$S_Y(T)\le C\gamma_Y(T)$.

We also obtain a convex-hull characterization of $S_Y(T)$.
Such a characterization was proved in \cite{LT} under the two
moment assumptions above. The following corollary answers the
question of Lata\l a \cite[Section~1]{LatalaConvexHull} for
log-concave tails.

\begin{corollary}\label{cor:intro-hull}
Put $\ell_j=\log(e+j)$, $j\ge1$. For every bounded nonempty $T\subset\R^n$, there are vectors
$s_j\in\R^n$, $j\ge1$, such that
\begin{equation}\label{eq:intro-hull}
 T-T\subset\conv\{\pm s_j:j\ge1\},
 \qquad
 \|\langle s_j,Y\rangle\|_{\ell_j}\le C S_Y(T).
\end{equation}
For every fixed $\theta\ge1$, \eqref{eq:intro-hull} holds with
$\ell_j$ replaced by $\theta\ell_j$ and $C$ by $C_\theta$.
\end{corollary}

If the inclusion in \eqref{eq:intro-hull} holds and
$\|Y_{s_j}\|_{\ell_j}\le M$ for all $j$, then
$S_Y(T)\le\E\sup_{j\ge1}|Y_{s_j}|$. Markov's inequality gives
$\E\sup_{j\ge1}|Y_{s_j}|\le CM$ (see Proposition~\ref{prop:hull-converse}).
The inclusion with these moment bounds need not hold for
log-convex tails \cite{HuWangConvexHull}. Section~\ref{sec:hulls} also gives
the closed convex hull result for arbitrary $T\subset\ell_2$
with $S_Y(T)<\infty$.

\subsection*{Outline of the proof}

The proof starts with a decomposition for the variables
\begin{equation}\label{eq:intro-cap}
 \xi_i^{(\beta_i)}=\varepsilon_i\min\{E_i,\beta_i\},
 \qquad 1\le\beta_i\le\infty.
\end{equation}
Write $\beta=(\beta_i)_i$, $\xi^{(\beta)}=(\xi_i^{(\beta_i)})_i$
and $S_\beta(T)=\E\sup_{t\in T}\langle t,\xi^{(\beta)}\rangle$,
where the sum defining the inner product is taken in $L^2$.
We shall prove the stronger estimate
\begin{equation}\label{eq:intro-cap-rem}
 S_\beta(T)\asymp
 \inf
 \left\{
 \sup_{x\in T_1}\sum_i\beta_i|x_i|
 +\gexp(T_2)
 \right\},
\end{equation}
where the infimum is as in \eqref{eq:intro-decomp}. We use the
convention $0\cdot\infty=0$.

For $1\le\beta\le\infty$, define $V_\beta$ on $[-\beta,\beta]$
(on $\R$ if $\beta=\infty$) by $V_\beta(u)=u^2$ for $|u|\le1$
and $V_\beta(u)=2|u|-1$ for $|u|>1$. In a constraint involving
$V_\beta(z)$ we require $z$ to belong to its domain. For $p>0$, put
\begin{equation}\label{eq:intro-rhop}
 C_p^\beta=\left\{z\in\ell_2:|z_i|\le\beta_i\ \text{for all }i,\
 \sum_i V_{\beta_i}(z_i)\le p\right\},\qquad
 \rho_p^\beta(t)=\sup_{z\in C_p^\beta}\langle t,z\rangle.
\end{equation}
For $p\ge2$, $\rho_p^\beta(t)\asymp\|\langle t,\xi^{(\beta)}\rangle\|_p$
by Lata\l a's moment estimate \cite[Theorem~1]{Latala96Mom}; see
Proposition~\ref{prop:est-moments}.

Since $V_{\beta_i}(u)\le2|u|$, we have
$\sum_{i:|x_i|>2\delta/q}\beta_i|x_i|\le\delta$ whenever
$\rho_q^\beta(x)\le\delta$ and $q\ge1$. For general tails the same bound with $b_i$ in place of $\beta_i$ need not hold.
We use this estimate to replace the coordinate intervals by
shorter ones.

We also use the following expression for $\gexp(S)$ when
$S\subset\ell_2$ is finite and nonempty:
\begin{equation}\label{eq:intro-exp-chain}
 \gexp(S)\asymp\inf\sup_{s\in S}\sum_{k=1}^N
 \left(2^{k/2}\|u_k(s)-u_{k-1}(s)\|_2
 +2^k\|u_k(s)-u_{k-1}(s)\|_\infty\right).
\end{equation}
The infimum is over $N\ge0$, admissible sequences $(\cA_k)_{0\le k\le N}$
with $\cA_N=\{\{s\}:s\in S\}$, and vectors $u_k(A)\in\ell_2$,
$A\in\cA_k$, with $u_N(\{s\})=s$. Here $u_k(s)=u_k(\cA_k(s))$.
See Lemma~\ref{lem:lower-gamma-exp}.

In Section~\ref{sec:estimates} we use Lata\l a's Sudakov
minoration and a concentration inequality to obtain a growth
condition for $\rho_p^\beta$. We also prove a form of Lata\l a's
principle (Proposition~\ref{prop:est-latala-principle}), which
will be used to remove some of the variables in the chopped process.

In Section~\ref{sec:chopping} we prove Theorem~\ref{thm:chop-main}.
We construct an increasing sequence of partitions $(\cA_n)$ of
$T$. For each $A\in\cA_n$ we choose a map $f_A:A\to\ell_2$, a
number $H(A)>0$ and intervals $Q_i(A)$ containing $f_{A,i}(A)$.
We denote by $F(A)$ the expected supremum of the process obtained
by applying the chopping maps given by $Q_i(A)\cap H(A)\Z$ to
$f_A(A)$. For $B\in\cA_{n+1}$, $B\subset A$, we have
$H(B)=H(A)$ or $H(B)=H(A)/R$, where $R$ is a sufficiently
large fixed integer. We keep $f_B=f_A$ on $B$ unless we replace
the intervals $Q_i(A)$ by shorter intervals, in which case
$f_B$ is obtained by coordinatewise projection.

When $H(B)=H(A)$, we obtain a bound for the expected supremum
on every subset of $B$ of sufficiently small diameter for the
chopping maps of length $H(A)/R$. We use this bound at later
steps, after the length has been reduced to $H(A)/R$, provided
that the intervals and the map have not changed. This gives
lower bounds for $F(A')-F(B)$, where $A'$ contains $B$ and
belongs to an earlier partition.
When the intervals are changed, Lata\l a's principle gives a
lower bound for $F(A)-F(B)$ which also controls the projection:
$\|f_B(t)-f_A(t)\|_{1,\beta}\le C(F(A)-F(B))$ for $t\in B$,
where $\|x\|_{1,\beta}=\sum_i\beta_i|x_i|$.
Writing $A_n=\cA_n(t)$, we obtain
\[
 \sum_{n\ge0}2^n H(A_n)\le C(F(T)+H(T)),\qquad
 \sum_{n\ge0}\|f_{A_{n+1}}(t)-f_{A_n}(t)\|_{1,\beta}\le CF(T).
\]
The first inequality allows us to choose the vectors $u_k(A)$
in Theorem~\ref{thm:chop-main}. The second bounds $\|t-z(t)\|_{1,\beta}$,
where $z(t)=\lim_n f_{A_n}(t)$. The constants are independent of $\beta$.

In Section~\ref{sec:lower} we use Theorem~\ref{thm:chop-main}
to bound the functional $\gamma^{\psi_\beta}$. Write
$\gamma^{\psi_\beta}$ for the functional defined in
\eqref{eq:lower-gamma-definition}, with $\psi_{\beta_i}$ the
function corresponding to $\xi_i^{(\beta_i)}$ in
\eqref{eq:intro-Psi}: $\psi_\beta(x)=x^2/4$ for $0\le x\le2$
and $\psi_\beta(x)=2\beta-1$ for $x>2$. The identity
$\int_0^\infty\psi_\beta(x)x^{-2}\,dx=\beta$ gives
$\gamma^{\psi_\beta}(T_1)\le C\sup_{t\in T_1}\|t\|_{1,\beta}$.
We also prove that
$\gamma^{\psi_\beta}(T_2)\le C\gexp(T_2)$. Combining the partitions of $T_1$ and $T_2$ gives
$\gamma^{\psi_\beta}(T)\le CS_\beta(T)$.

To pass to general tails we use the construction of \cite{HWW},
where the values of $\sigma_i$ are grouped in the intervals
$[2^{j+1},2^{j+2})$. Let $\beta_{ij}$ be as in \eqref{eq:lower-betaij}, put
$\widetilde\beta=((1)_i,(\beta_{ij})_{i,j})$ and
$Lt=(t,(2^{-j}t_i)_{i,j})$. Comparing $\Psi_i$ with
$x\mapsto\psi_1(x)+\sum_j\psi_{\beta_{ij}}(2^{-j}x)$, and
using convex order for the comparison of the expected suprema,
we obtain
\[
 \gamma^\Psi(T)\le3\gamma^{\psi_{\widetilde\beta}}(LT)
 \le CS_{\widetilde\beta}(LT)\le CS_Y(T).
\]
Here $\gamma^\Psi$ is defined in \eqref{eq:lower-gamma-definition}
with $\theta=\Psi$, and $\gamma^\Psi(T)\le\gamma_Y(T)\le r\gamma^\Psi(T)$.

It remains to construct the decomposition of $T$ from these
partitions. In Section~\ref{sec:upper} we choose $(\cA_n,r_n)$
as in \eqref{eq:lower-gamma-definition} with
$B=\sup_{t\in T}\sum_{n\ge0}2^nr_n(\cA_n(t))\le2\gamma^\Psi(T)$,
so $B\le CS_Y(T)$.
Choose $\pi_n(A)\in A$ for $A\in\cA_n$ and put $t_0=\pi_0(T)$.
For $A\in\cA_n$, $n\ge1$, let $A'\in\cA_{n-1}$ contain $A$.
All quantities indexed by $A$ below refer to $(n,A)$. Put $z_A=\pi_n(A)-\pi_{n-1}(A')$, $r_A=r_{n-1}(A')$ and
$h_{A,i}=\operatorname{sgn}(z_{A,i})\min\{|z_{A,i}|,4r_A\}$.
We set $v(t)=\sum_{n\ge1}h_{\cA_n(t)}$ and $a(t)=t-t_0-v(t)$.
The diameter conditions give $\gexp(v(T))\le CB$.

To estimate $a(t)$, we also consider the vectors $a^{(\kappa)}(t)$
defined for $\kappa\ge2$ by replacing $4r_A$ by $2\kappa r_A$
in the definition of $h_{A,i}$. Applying Lemma~\ref{lem:upper-scalar} on each
coordinate and integrating in $u$ using \eqref{eq:intro-Psi},
we obtain
\[
 \sum_i|a_i(t)|+\sum_i\int_1^{b_i}
 |a_i^{(\sigma_i(u))}(t)|\,du\le CB.
\]
To estimate $\E\sup_{t\in T}\sum_i|a_i(t)|\,|Y_i|$, we compare
the sums defining $a_i^{(\sigma_i(u))}(t)$ and $a_i(t)$.
For each $A$ only the indices $i\in I_A=\{i:|z_{A,i}|>4r_A\}$
contribute to the difference. Integrating in $u$ up to
$q_i(E_i)\vee1$ and summing over $i$ gives a bound by
$4r_A\sum_{i\in I_A}E_i$. Since $|I_A|\le2^{n-1}$ and
$|\cA_n|\le2^{2^n}$, we have
$\E\sup_{n\ge1,\,A\in\cA_n}2^{-n}\sum_{i\in I_A}E_i\le C$.
This gives $\E\sup_{t\in T}\sum_i|a_i(t)|\,|Y_i|\le CB$.

In Section~\ref{sec:hulls} we use the vectors $h_A$ and finite
nets in the polars of convex bodies on $\R^{I_A}$ to construct
the vectors $s_j$ in Corollary~\ref{cor:intro-hull}.

\section{Estimates for \texorpdfstring{$\xi^{(\beta)}$}{capped exponentials}}\label{sec:estimates}

Let $\N=\{1,2,\ldots\}$. For $I\subset\N$ and $t\in\ell_2$, put
$(t_I)_i=t_i\ind_{\{i\in I\}}$ and
\begin{equation}\label{eq:est-XI}
 X_t^I=\sum_{i\in I}t_i\xi_i^{(\beta_i)}\quad\text{in }L^2,
 \qquad S_I(T)=\E\sup_{t\in T}X_t^I.
\end{equation}
Thus $X_t^\N=\langle t,\xi^{(\beta)}\rangle$ and $S_\N(T)=S_\beta(T)$.
For $p>0$ let
\begin{equation}\label{eq:est-rhoI}
 C_{p,I}^\beta=\{z_I:z\in C_p^\beta\},
 \qquad
 \rho_{p,I}^\beta(t)=\sup_{z\in C_{p,I}^\beta}\langle t,z\rangle
 =\rho_p^\beta(t_I).
\end{equation}
For $p\ge1$, $C_{p,I}^\beta$ contains $\{z:z=z_I,\ \|z\|_2\le1\}$, so
$\|t_I\|_2\le\rho_{p,I}^\beta(t)$. For a seminorm $\rho$, put
$\Delta(A,\rho)=\sup_{s,t\in A}\rho(s-t)$.

\subsection{\texorpdfstring{The distances $\rho_{p,I}^\beta$}{Moment distances}}

\begin{lemma}\label{lem:est-scaling}
For $q\ge p>0$ we have
\begin{equation}\label{eq:est-scaling}
 \rho_{p,I}^\beta(t)\le\rho_{q,I}^\beta(t)
 \le\frac{q}{p}\rho_{p,I}^\beta(t).
\end{equation}
Moreover,
\begin{equation}\label{eq:est-l1}
 \rho_{p,I}^\beta(t)\le\sum_{i\in I}\beta_i|t_i|.
\end{equation}
We also have
\begin{equation}\label{eq:est-rho-infty}
 \rho_{p,I}^\beta(t)\le\rho_{p,I}^\infty(t)
 \le\sqrt p\,\|t_I\|_2+p\|t_I\|_\infty,
\end{equation}
where $\infty$ denotes the sequence with all entries equal to $\infty$.
\end{lemma}

\begin{proof}
By convexity and $V_{\beta_i}(0)=0$,
$V_{\beta_i}((p/q)z_i)\le(p/q)V_{\beta_i}(z_i)$, so
\[
 C_{p,I}^\beta\subset C_{q,I}^\beta\subset(q/p)C_{p,I}^\beta.
\]
Taking suprema gives \eqref{eq:est-scaling}. For \eqref{eq:est-l1},
use $t_i z_i\le |t_i|\beta_i$.

Since $C_{p,I}^\beta\subset C_{p,I}^\infty$, it suffices to bound
$\rho_{p,I}^\infty(t)$. Take $z\in C_{p,I}^\infty$
and set $I_0=\{i\in I:|z_i|\le1\}$, $I_1=I\setminus I_0$.
Then
\[
 \sum_{i\in I_0}z_i^2\le p,
 \qquad \sum_{i\in I_1}|z_i|\le p.
\]
Hence
\[
 \sum_{i\in I}t_i z_i
 \le\sqrt p\,\|t_I\|_2+p\|t_I\|_\infty.
\]
Taking the supremum over $z$ gives \eqref{eq:est-rho-infty}.
\end{proof}

We will repeatedly use the following two inequalities.

\begin{lemma}\label{lem:est-deletion-diameter}
If $I\subset J\subset\N$, then
\begin{equation}\label{eq:est-deletion}
 S_I(T)\le S_J(T).
\end{equation}
Moreover,
\begin{equation}\label{eq:est-diam2}
 \Delta(T,d_2)\le C S_\beta(T).
\end{equation}
\end{lemma}

\begin{proof}
For finite $T$, integrate first over the variables indexed by $J\setminus I$.
With $\E_{J\setminus I}$ denoting this integral, we have
\[
 \E_{J\setminus I}\sup_{t\in T}X_t^J
 \ge\sup_{t\in T}\E_{J\setminus I}X_t^J
 =\sup_{t\in T}X_t^I.
\]
Integrate over the variables indexed by $I$ and use \eqref{eq:intro-width}.
For $s,t\in T$, since $\E X_t^\N=\E X_s^\N=0$,
\[
 S_\beta(T)\ge\E\max\{X_s^\N,X_t^\N\}
 =\frac12\E|X_{s-t}^\N|.
\]
Apply the Khintchine inequality to the signs and then Jensen's
inequality to the Euclidean norm. This gives
\[
\begin{aligned}
 \E|X_{s-t}^\N|
 &\ge c\E\Bigl(\sum_i(s_i-t_i)^2
                 \min\{E_i,\beta_i\}^2\Bigr)^{1/2}\\
 &\ge c\Bigl(\sum_i(s_i-t_i)^2
                 (\E\min\{E_i,\beta_i\})^2\Bigr)^{1/2}\\
 &\ge c\|s-t\|_2,
\end{aligned}
\]
since $\E\min\{E_i,\beta_i\}\ge1-e^{-1}$.
\end{proof}

The following is a consequence of the moment estimate of
Lata\l a \cite[Theorem~1]{Latala96Mom}.

\begin{proposition}\label{prop:est-moments}
For $p\ge2$, $I\subset\N$ and $t\in\ell_2$,
\begin{equation}\label{eq:est-rho-lp}
 c\|X_t^I\|_p\le\rho_{p,I}^\beta(t)
 \le C\|X_t^I\|_p.
\end{equation}
Consequently,
\begin{equation}\label{eq:est-bernstein}
 \|X_t^I\|_p
 \le C\bigl(\sqrt p\,\|t_I\|_2+p\|t_I\|_\infty\bigr).
\end{equation}
\end{proposition}

\begin{proof}
Assume first that $I$ is finite. For $|u|\le\beta$, put
\[
 \widehat V_\beta(u)=
 \begin{cases}
  u^2,&|u|\le1,\\
  |u|,&|u|>1.
 \end{cases}
\]
For the $i$th variable in \eqref{eq:intro-cap},
\[
 -\log\P\bigl(|\xi_i^{(\beta_i)}|\ge u\bigr)
 =
 \begin{cases}
 u,&0\le u\le\beta_i,\\
 +\infty,&u>\beta_i.
 \end{cases}
\]
Thus $\widehat V_{\beta_i}$, extended by $+\infty$ off its domain,
is precisely the function $\widehat N_i$ of \cite{Latala96Mom}.
Write
\[
 \widehat\rho_{p,I}^\beta(t)=
 \sup\left\{\sum_{i\in I}t_i z_i:
 |z_i|\le\beta_i\ \text{for }i\in I,\
 \sum_{i\in I}\widehat V_{\beta_i}(z_i)\le p\right\}.
\]
Since $\widehat V_\beta\le V_\beta\le2\widehat V_\beta$,
Lemma~\ref{lem:est-scaling} gives
\begin{equation}\label{eq:est-moment-comparison}
 \rho_{p,I}^\beta(t)
 \le\widehat\rho_{p,I}^\beta(t)
 \le\rho_{2p,I}^\beta(t)
 \le2\rho_{p,I}^\beta(t).
\end{equation}
By \cite[Theorem~1]{Latala96Mom},
\[
 \|X_t^I\|_p\asymp
 \|X_t^I\|_1+\widehat\rho_{p,I}^\beta(t).
\]
Also $\E(\xi_i^{(\beta_i)})^2\le2$, so
\[
 \|X_t^I\|_1\le\sqrt2\,\|t_I\|_2
 \le\sqrt2\,\rho_{p,I}^\beta(t).
\]
Together with \eqref{eq:est-moment-comparison} this proves
\eqref{eq:est-rho-lp}. Finally,
$C_{p,I}^\beta\subset C_{p,I}^\infty$, and
\eqref{eq:est-bernstein} follows from
\eqref{eq:est-rho-infty}.

Put $I_n=I\cap\{1,\ldots,n\}$. Since $X_t^{I_m}\to X_t^I$ in $L^2$,
Fatou's lemma (applied along a subsequence) and
\eqref{eq:est-bernstein} for finite index sets give
\[
 \|X_t^I-X_t^{I_n}\|_p
 \le\liminf_{m\to\infty}\|X_t^{I_m\setminus I_n}\|_p
 \le Cp\|t_{I\setminus I_n}\|_2.
\]
By \eqref{eq:est-rhoI} and \eqref{eq:est-rho-infty},
\[
 0\le\rho_{p,I}^\beta(t)-\rho_{p,I_n}^\beta(t)
 \le\rho_p^\beta(t_{I\setminus I_n})\le2p\|t_{I\setminus I_n}\|_2.
\]
Let $n\to\infty$ in \eqref{eq:est-rho-lp} and \eqref{eq:est-bernstein}
for $I_n$.
\end{proof}

\subsection{Sudakov minoration and concentration}

We recall the Sudakov minoration of Lata\l a
\cite[Theorem~1]{LatalaSudakov}; see also \cite[Theorem~1.3]{LT}.

\begin{theorem}\label{thm:est-sudakov}
Let $I\subset\N$, $p\ge2$, $m\ge e^p$, $\sigma>0$,
and let $t_1,\ldots,t_m\in\ell_2$. If
\begin{equation}\label{eq:est-lp-sep}
 \|X_{t_l}^I-X_{t_k}^I\|_p\ge\sigma,
 \qquad l\ne k,
\end{equation}
then
\begin{equation}\label{eq:est-sudakov-lp}
 \E\max_{1\le l\le m}X_{t_l}^I\ge c\sigma.
\end{equation}
\end{theorem}

\begin{proof}
By \eqref{eq:est-bernstein}, there is a finite $J\subset I$
such that $\|X_{t_l}^J-X_{t_k}^J\|_p\ge\sigma/2$ for $l\ne k$.
Apply Lata\l a's minoration on $J$ and use \eqref{eq:est-deletion}.
\end{proof}

Proposition~\ref{prop:est-moments} gives the form we shall use.

\begin{corollary}\label{cor:est-sudakov-rho}
Let $I\subset\N$, $p\ge2$, $m\ge e^p$, $\sigma>0$,
and let $t_1,\ldots,t_m\in\ell_2$. If
\begin{equation}\label{eq:est-rho-sep}
 \rho_{p,I}^\beta(t_l-t_k)\ge\sigma,
 \qquad l\ne k,
\end{equation}
then
\begin{equation}\label{eq:est-sudakov-rho}
 S_I(\{t_1,\ldots,t_m\})\ge c\sigma.
\end{equation}
\end{corollary}

We shall use the following consequence of Talagrand's
isoperimetric inequality \cite{TalagrandIso}.

\begin{theorem}\label{thm:est-affine}
Let $I\subset\N$, let $A\subset\ell_2$ be nonempty and countable,
and let $b_t\in\R$ for $t\in A$. Put
\begin{equation}\label{eq:est-affine-f}
 Z=\sup_{t\in A}(b_t+X_t^I).
\end{equation}
If $Z$ is integrable, then for $p\ge2$,
\begin{equation}\label{eq:est-affine-bound}
 \|Z-\E Z\|_p\le C\sup_{t\in A}\rho_{p,I}^\beta(t).
\end{equation}
\end{theorem}

\begin{proof}
Suppose first that $A$ and $I$ are finite. Identify
$\R^I=\{x\in\ell_2:x=x_I\}$ and put
\[
 f(x)=\max_{t\in A}\left(b_t+\sum_{i\in I}t_i x_i\right),\qquad
 \Theta_\beta(x)_i=\operatorname{sgn}(x_i)
                  \min\{|x_i|,\beta_i\}.
\]
Let $\nu$ be the symmetric exponential measure. For $x\sim\nu^{\otimes I}$,
$f(\Theta_\beta(x))$ has the law of $Z$. By Talagrand's inequality
\cite[proof of Theorem~1]{Latala96Mom}, if $\nu^{\otimes I}(B)\ge1/2$, then
\begin{equation}\label{eq:est-exp-enlargement}
 \nu^{\otimes I}(B+U_s)\ge1-2e^{-s},\qquad
 U_s=\left\{z\in\R^I:\sum_{i\in I}\min\{z_i^2,|z_i|\}\le36s\right\}.
\end{equation}
For $x-y\in U_s$, set $u=\Theta_\beta(x)-\Theta_\beta(y)$.
Since $|u_i|\le|x_i-y_i|$ and $|u_i|\le2\beta_i$,
\[
 V_{\beta_i}(u_i/2)
 \le\min\{|x_i-y_i|^2,|x_i-y_i|\}.
\]
Hence $u/2\in C_{36s,I}^\beta$, and
\begin{equation}\label{eq:est-affine-increment}
 |f(\Theta_\beta(x))-f(\Theta_\beta(y))|
 \le2\sup_{t\in A}\rho_{36s,I}^\beta(t).
\end{equation}
Let $M$ be a median of $Z$ and put
$B_- =\{x:f(\Theta_\beta(x))\le M\}$ and
$B_+=\{x:f(\Theta_\beta(x))\ge M\}$. Put
$R_s=2\sup_{t\in A}\rho_{36s,I}^\beta(t)$. By
\eqref{eq:est-affine-increment},
\[
 B_-+U_s\subset\{x:f(\Theta_\beta(x))\le M+R_s\},\qquad
 B_++U_s\subset\{x:f(\Theta_\beta(x))\ge M-R_s\}.
\]
Thus $\P(|Z-M|>R_s)\le4e^{-s}$ by \eqref{eq:est-exp-enlargement}.
Put $R=\sup_{t\in A}\rho_{p,I}^\beta(t)$. By \eqref{eq:est-scaling},
\[
 R_s\le72(s/p)R,\qquad s\ge p.
\]
For $R>0$, put $U=|Z-M|/(72R)$. Then $\P(U>u)\le4e^{-pu}$ for $u\ge1$ and
\[
 \E U^p\le1+4p\int_1^\infty u^{p-1}e^{-pu}\,du
 \le1+4\frac{\Gamma(p+1)}{p^p}.
\]
Hence $\|Z-\E Z\|_p\le2\|Z-M\|_p\le CR$. If $R=0$, then $t_I=0$
for every $t\in A$.

For finite $A$ and arbitrary $I$, set
$I_n=I\cap\{1,\ldots,n\}$ and $Z_n=\max_{t\in A}(b_t+X_t^{I_n})$.
By \eqref{eq:est-bernstein},
\[
 \|Z-Z_n\|_p\le\sum_{t\in A}\|X_t^{I\setminus I_n}\|_p
 \le Cp\sum_{t\in A}\|t_{I\setminus I_n}\|_2\longrightarrow0.
\]
Since $\rho_{p,I_n}^\beta(t)\le\rho_{p,I}^\beta(t)$, we have
\[
 \|Z-\E Z\|_p\le\|Z_n-\E Z_n\|_p+2\|Z-Z_n\|_p
 \le C\sup_{t\in A}\rho_{p,I}^\beta(t)+2\|Z-Z_n\|_p.
\]
For countable $A$, take finite nonempty $A_k\uparrow A$ and put
$W_k=\max_{t\in A_k}(b_t+X_t^I)$. Then $W_k\uparrow Z$ and
$\E(W_k-W_1)\uparrow\E(Z-W_1)$. By Fatou's lemma,
\[
 \|Z-\E Z\|_p\le\liminf_{k\to\infty}\|W_k-\E W_k\|_p
 \le C\sup_{t\in A}\rho_{p,I}^\beta(t).
\]
\end{proof}

\begin{corollary}\label{cor:est-max-fluct}
Let $m\ge8$, $p=\log m$, and let $B_1,\ldots,B_m\subset\ell_2$
be finite and nonempty. Choose $t_j\in B_j$ and assume that
\begin{equation}\label{eq:est-local-radius}
 \sup_{t\in B_j}\rho_p^\beta(t-t_j)\le v,
 \qquad 1\le j\le m.
\end{equation}
Put
\[
 Z_j=\sup_{t\in B_j}\langle t-t_j,\xi^{(\beta)}\rangle.
\]
Then
\begin{equation}\label{eq:est-max-fluct}
 \E\max_{1\le j\le m}|Z_j-\E Z_j|\le Cv.
\end{equation}
\end{corollary}

\begin{proof}
By Theorem~\ref{thm:est-affine},
$\|Z_j-\E Z_j\|_p\le Cv$. Hence
\[
 \E\max_{j\le m}|Z_j-\E Z_j|
 \le\Bigl(\sum_{j=1}^m\E|Z_j-\E Z_j|^p\Bigr)^{1/p}
 \le Cm^{1/p}v=Cev.
\]
\end{proof}

\subsection{The growth condition}

Combining the preceding two corollaries gives the growth condition.
See \cite[Proposition~2.8]{BL} and
\cite[Corollary~3]{LatalaSudakov}.

\begin{lemma}\label{lem:est-growth}
Let $m\ge8$, $p=\log m$, $\sigma>0$, $v\ge0$,
and let $t_1,\ldots,t_m\in\ell_2$ satisfy
\begin{equation}\label{eq:est-growth-separation}
 \rho_p^\beta(t_k-t_l)\ge\sigma,\qquad k\ne l.
\end{equation}
Let $B_1,\ldots,B_m\subset\ell_2$ be nonempty sets such that
$t_k\in B_k$ and
\begin{equation}\label{eq:est-growth-radius}
 \sup_{t\in B_k}\rho_p^\beta(t-t_k)\le v,
 \qquad k\le m.
\end{equation}
Then
\begin{equation}\label{eq:est-growth}
 S_\beta\left(\bigcup_{k=1}^m B_k\right)
 \ge \min_{k\le m}S_\beta(B_k)+c\sigma-Cv.
\end{equation}
\end{lemma}

\begin{proof}
By \eqref{eq:intro-width} it suffices to prove the inequality for
finite $B_k$ containing $t_k$. Put
\[
 Z_k=\sup_{t\in B_k}\langle t-t_k,\xi^{(\beta)}\rangle.
\]
Since $\E\langle t_k,\xi^{(\beta)}\rangle=0$,
\[
 \E Z_k
 =\E\sup_{t\in B_k}\langle t,\xi^{(\beta)}\rangle
 =S_\beta(B_k).
\]
Furthermore,
\[
\begin{aligned}
 \sup_{t\in\bigcup_k B_k}\langle t,\xi^{(\beta)}\rangle
 &=\max_k\bigl(\langle t_k,\xi^{(\beta)}\rangle+Z_k\bigr)\\
 &\ge\max_k\langle t_k,\xi^{(\beta)}\rangle
       +\min_k\E Z_k-\max_k|Z_k-\E Z_k|.
\end{aligned}
\]
By Corollaries~\ref{cor:est-sudakov-rho} and
\ref{cor:est-max-fluct}, respectively,
\[
 \E\max_k\langle t_k,\xi^{(\beta)}\rangle\ge c\sigma,
 \qquad \E\max_k|Z_k-\E Z_k|\le Cv.
\]
Taking expectations proves \eqref{eq:est-growth}.
\end{proof}

\subsection{Lata\l a's principle}

For $p>0$ and $\sigma>0$ write
\[
 B_{p,I}(t,\sigma)
 =\{s\in\ell_2:\rho_{p,I}^\beta(s-t)\le\sigma\}.
\]
We use the following form of Lata\l a's principle. See
\cite[Proposition~2.10]{BL} and
\cite[Proposition~10.1.1]{TalagrandBook}.

\begin{proposition}\label{prop:est-latala-principle}
Let $J\subset I\subset\N$, $m\ge8$,
$p=\log m$, $\delta\ge0$, and let $T\subset\ell_2$
be nonempty. Suppose that
\begin{equation}\label{eq:est-latala-small}
 \Delta(T,\rho_{p,J}^\beta)\le\delta.
\end{equation}
For every $\sigma>0$, there is a partition of $T$ into at most $m$
sets satisfying
\begin{equation}\label{eq:est-latala-cover}
 A_k\subset B_{p,I}(t_k,\sigma)
 \quad\text{for some }t_k\in T,
\end{equation}
and at most one further set $A$ satisfying
\begin{equation}\label{eq:est-latala-drop}
 S_J(A)\le S_I(T)-c\sigma+C\delta.
\end{equation}
\end{proposition}

\begin{proof}
If $S_I(T)=\infty$, take $A=T$. If $S_I(T)<\infty$ and
$\sigma\le2\delta$, take $A=T$, since
\[
 S_J(T)\le S_I(T)\le S_I(T)-c\sigma+2c\delta.
\]
Assume $S_I(T)<\infty$ and $\sigma>2\delta$. By translation we may
suppose $0\in T$, so
\begin{equation}\label{eq:est-latala-radius}
 \rho_{p,J}^\beta(t)\le\delta,\qquad t\in T.
\end{equation}
Put $B(t)=B_{p,I}(t,\sigma)$. For $t_1,\ldots,t_m\in T$, put
\[
 R_k=T\setminus\bigcup_{l<k}B(t_l),\qquad
 A_k=R_k\cap B(t_k)\quad(1\le k\le m),\qquad A_{m+1}=R_{m+1}.
\]
If $R_{m+1}=\emptyset$ for some choice of $t_1,\ldots,t_m$,
the nonempty sets $A_k$, $k\le m$, give the desired partition.
Otherwise put
\[
 \alpha=\inf_{t_1,\ldots,t_m\in T} S_J(R_{m+1}).
\]

Fix $\varepsilon>0$. Choose a finite nonempty $F_1\subset T$
with $S_J(F_1)\ge\alpha-\varepsilon$. Inductively, for $2\le k\le m$ and each choice of
$t_j\in F_j(t_1,\ldots,t_{j-1})$, $1\le j<k$ (with $F_1$ fixed), choose a finite nonempty set
\[
 F_k(t_1,\ldots,t_{k-1})\subset R_k,\qquad
 S_J(F_k(t_1,\ldots,t_{k-1}))\ge\alpha-\varepsilon.
\]
This is possible since $S_J(R_k)\ge\alpha$. Let $F$ be the union
of these finitely many finite sets. Write $F_k$ for
$F_k(t_1,\ldots,t_{k-1})$ when the points have been chosen.

Let $(\widetilde X_t^{(k)})_{t\in F}$, $1\le k\le m$, be independent
copies of $(X_t^J)_{t\in F}$, independent of $(Z_t)_{t\in F}$, where
$Z_t=X_t^{I\setminus J}$. For $k=1,\ldots,m$, take
\[
 W_k=\max_{t\in F_k}\widetilde X_t^{(k)},\qquad
 t_k\in F_k,\qquad \widetilde X_{t_k}^{(k)}=W_k,
\]
using a fixed order on each $F_k$. Since $t_k\in R_k$,
\begin{equation}\label{eq:est-latala-separation}
 \rho_{p,I}^\beta(t_k-t_l)>\sigma,\qquad l<k\le m.
\end{equation}
For each possible value $D$ of $F_k$, the event $\{F_k=D\}$ is
independent of the $k$th copy. By \eqref{eq:est-affine-bound} and
\eqref{eq:est-latala-radius},
\[
 \E\left|\max_{t\in D}\widetilde X_t^{(k)}-S_J(D)\right|^p
 \le(C\delta)^p.
\]
Multiplying by $\P(F_k=D)$ and summing over $D$ gives
$\E|W_k-S_J(F_k)|^p\le(C\delta)^p$.
Since $m^{1/p}=e$,
\begin{equation}\label{eq:est-latala-min}
 \E\min_{k\le m}W_k
 \ge\alpha-\varepsilon-
 \left(\sum_{k\le m}\E|W_k-S_J(F_k)|^p\right)^{1/p}
 \ge\alpha-\varepsilon-C\delta.
\end{equation}

By \eqref{eq:est-rhoI},
\begin{equation}\label{eq:est-latala-triangle}
 \rho_{p,I}^\beta(u)
 =\rho_p^\beta(u_J+u_{I\setminus J})
 \le\rho_{p,J}^\beta(u)+\rho_{p,I\setminus J}^\beta(u).
\end{equation}
For $k\ne l$, \eqref{eq:est-latala-separation} and
\eqref{eq:est-latala-small} give
$\rho_{p,I\setminus J}^\beta(t_k-t_l)>\sigma-\delta>\sigma/2$.
The points $t_1,\ldots,t_m$ are independent of $(Z_t)_{t\in F}$.
For each possible choice of these points, apply Corollary~\ref{cor:est-sudakov-rho}
in the variables indexed by $I\setminus J$. This gives
\begin{equation}\label{eq:est-latala-Z}
 \E\max_{k\le m}Z_{t_k}\ge c_1\sigma.
\end{equation}
Consequently, with $S=\E\max_{k\le m}(W_k+Z_{t_k})$,
\[
 S\ge \E\min_{k\le m}W_k+\E\max_{k\le m}Z_{t_k}
 \ge\alpha+c_1\sigma-C\delta-\varepsilon.
\]

For $z=(z_t)_{t\in F}\in\R^F$, put
\[
 H_k(z)=\max_{t\in F}(\widetilde X_t^{(k)}+z_t),\qquad
 g(z)=\E H_1(z).
\]
By \eqref{eq:est-affine-bound}, $\E|H_k(z)-g(z)|^p\le(C\delta)^p$.
Hence
\[
 \E\max_{k\le m}H_k(z)
 \le g(z)+\left(\sum_{k\le m}\E|H_k(z)-g(z)|^p\right)^{1/p}
 \le g(z)+C\delta.
\]
Integrate with respect to the law of $Z=(Z_t)_{t\in F}$, which is
independent of the copies. Since $\E g(Z)=S_I(F)$, we obtain
\begin{equation}\label{eq:est-latala-H}
 S\le\E\max_{k\le m}H_k(Z)\le S_I(F)+C\delta\le S_I(T)+C\delta.
\end{equation}
Thus $\alpha\le S_I(T)-c_1\sigma+C\delta+\varepsilon$.
Take $\varepsilon=c_1\sigma/4$ and choose $t_1,\ldots,t_m\in T$
with $S_J(R_{m+1})\le\alpha+\varepsilon$. The nonempty sets $A_k$, $k\le m$, satisfy
\eqref{eq:est-latala-cover}, and $A=R_{m+1}$ satisfies
\eqref{eq:est-latala-drop} with $c=c_1/2$.
\end{proof}

\section{A decomposition theorem}\label{sec:chopping}

We first prove the decomposition for the variables $\xi_i^{(\beta_i)}$
of \eqref{eq:intro-cap}.

For $x\in\ell_2$ put
\begin{equation}\label{eq:chop-l1beta}
 \|x\|_{1,\beta}=\sum_i\beta_i|x_i|,
\end{equation}
with $0\cdot\infty=0$.

\begin{theorem}\label{thm:chop-main}
Let $T\subset\ell_2$ be nonempty and $S_\beta(T)<\infty$.
There are an admissible sequence $(\cA_k)_{k\ge0}$ of
partitions of $T$, vectors $u_k(A)\in\ell_2$, $A\in\cA_k$,
and a map $z:T\to\ell_2$ such that, with $u_k(t)=u_k(\cA_k(t))$,
\[
 z(t)=u_0(T)+\sum_{k\ge1}\bigl(u_k(t)-u_{k-1}(t)\bigr).
\]
We have
\begin{equation}\label{eq:chop-main}
 \sup_{t\in T}\|t-z(t)\|_{1,\beta}
 +\sup_{t\in T}\sum_{k\ge1}
 \rho_{2^k}^\beta\bigl(u_k(t)-u_{k-1}(t)\bigr)
 \le C S_\beta(T).
\end{equation}
Consequently there are $a,v:T\to\ell_2$ and $t_0\in T$ such that
\[
 t-t_0=a(t)+v(t)
\]
and
\begin{equation}\label{eq:chop-main-decomp}
 \sup_{t\in T}\|a(t)\|_{1,\beta}+\gexp(v(T))\le C S_\beta(T).
\end{equation}
\end{theorem}

We construct the partitions in Section~\ref{sec:partitions} by
repeated applications of Lemma~\ref{lem:chop-three}.

\subsection{Chopping maps}

We use Talagrand's chopping maps. For $u<v$ define
\begin{equation}\label{eq:chop-phi}
 \varphi_{u,v}(x)=\min\{v,\max\{x,u\}\}-\min\{v,\max\{0,u\}\}.
\end{equation}
Thus $\varphi_{u,v}(0)=0$ and $|\varphi_{u,v}(x)|\le v-u$.
If $u_0<u_1<\cdots<u_m$, then
\begin{equation}\label{eq:chop-sum}
 \varphi_{u_0,u_m}(x)=\sum_{j=1}^m\varphi_{u_{j-1},u_j}(x).
\end{equation}
For these chopping maps see \cite{Talagrand93,Talagrand94Bernoulli},
\cite[Section~4]{BL} and \cite[Section~10.3.1]{TalagrandBook}.

\begin{lemma}\label{lem:chop-basic}
For $u_0<u_1<\cdots<u_m$ and $x,y\in\R$,
\begin{equation}\label{eq:chop-variation}
 \sum_{j=1}^m
 \bigl|\varphi_{u_{j-1},u_j}(x)-\varphi_{u_{j-1},u_j}(y)\bigr|
 =|\varphi_{u_0,u_m}(x)-\varphi_{u_0,u_m}(y)|
 \le|x-y|.
\end{equation}
In particular,
\begin{equation}\label{eq:chop-square}
 \sum_{j=1}^m|\varphi_{u_{j-1},u_j}(x)|^2\le x^2.
\end{equation}
\end{lemma}

\begin{proof}
Put $I=[x\wedge y,x\vee y]$, and write $|J|$ for the length of an
interval $J$. By \eqref{eq:chop-phi},
\[
 \bigl|\varphi_{u,v}(x)-\varphi_{u,v}(y)\bigr|=|[u,v]\cap I|.
\]
Summing over $j$ gives \eqref{eq:chop-variation}. Taking $y=0$ in
\eqref{eq:chop-variation}, we obtain
\[
 \sum_{j=1}^m|\varphi_{u_{j-1},u_j}(x)|^2
 \le\left(\sum_{j=1}^m|\varphi_{u_{j-1},u_j}(x)|\right)^2\le x^2.
\]
\end{proof}

For each $i\ge1$ let $G_i\subset\R$ be finite and nonempty.
Put $G=(G_i)_{i\ge1}$. We write
\[
 G_i^-=G_i\setminus\{\max G_i\},\qquad
 u^+=\min\{v\in G_i:u<v\}\quad(u\in G_i^-).
\]
Put
\[
 I_G=\{(i,u):i\ge1,\ u\in G_i^-\}.
\]
For $t\in\ell_2$, define
\begin{equation}\label{eq:chop-PhiG}
 \Phi_G(t)=
 \bigl(\varphi_{u,u^+}(t_i)\bigr)_{(i,u)\in I_G}.
\end{equation}
By \eqref{eq:chop-variation},
\[
 \|\Phi_G(t)-\Phi_G(s)\|_2^2
 \le \sum_i\left(\sum_{u\in G_i^-}
       |\varphi_{u,u^+}(t_i)-\varphi_{u,u^+}(s_i)|\right)^2
 \le\|t-s\|_2^2.
\]
In particular $\Phi_G(t)\in\ell_2(I_G)$.
For $(i,u)\in I_G$, take independent variables
$\xi_{i,u}^{(\beta_i)}$ having the law of $\xi_i^{(\beta_i)}$, and put
\begin{equation}\label{eq:chop-process}
 X_t(G)=
 \sum_{(i,u)\in I_G}
 \varphi_{u,u^+}(t_i)\xi_{i,u}^{(\beta_i)}\quad\text{in }L^2.
\end{equation}
For $\emptyset\ne A\subset\ell_2$ put
\begin{equation}\label{eq:chop-functional}
 F_G(A)=\E\sup_{t\in A}X_t(G).
\end{equation}
Put
\begin{equation}\label{eq:chop-rhoG}
 \begin{aligned}
 C_p^G
 &=\left\{z\in\ell_2(I_G): |z_{i,u}|\le\beta_i,\
                  \sum_{(i,u)\in I_G}V_{\beta_i}(z_{i,u})\le p\right\},\\
 \rho_p^G(\zeta)&=\sup_{z\in C_p^G}\sum_{(i,u)\in I_G}\zeta_{i,u}z_{i,u}.
 \end{aligned}
\end{equation}
For $s,t\in\ell_2$, put
\[
 d_{p,G}(s,t)=\rho_p^G\bigl(\Phi_G(t)-\Phi_G(s)\bigr).
\]
The next lemma compares $F_G$ and $d_{p,G}$ before and after
adding points to $G_i$. See
\cite[Propositions~10.3.9--10.3.10]{TalagrandBook} and
\cite[Propositions~4.2--4.3]{BL}.

\begin{lemma}\label{lem:chop-refine}
Let $G_i\subset G_i'$ be finite and nonempty, and put
$G=(G_i)_i$ and $G'=(G_i')_i$. Assume that
\begin{equation}\label{eq:chop-same-endpoints}
 \min G_i=\min G_i',
 \qquad
 \max G_i=\max G_i'.
\end{equation}
Then, for every nonempty $A\subset\ell_2$,
\begin{equation}\label{eq:chop-refine-F}
 F_{G'}(A)\le F_G(A).
\end{equation}
For $s,t\in\ell_2$ and $p>0$, we have
\begin{equation}\label{eq:chop-refine-rho}
 \rho_p^{G'}\bigl(\Phi_{G'}(t)-\Phi_{G'}(s)\bigr)
 \le \rho_p^G\bigl(\Phi_G(t)-\Phi_G(s)\bigr).
\end{equation}
If, for some integer $K\ge1$,
\begin{equation}\label{eq:chop-refine-card}
 \#\bigl(G_i'^-\cap[u,u^+)\bigr)\le K,
 \qquad
 i\ge1,\quad u\in G_i^-,
\end{equation}
then also
\[
 \rho_p^G\bigl(\Phi_G(t)-\Phi_G(s)\bigr)
 \le
 K\rho_p^{G'}\bigl(\Phi_{G'}(t)-\Phi_{G'}(s)\bigr).
\]
\end{lemma}

\begin{proof}
Put
\[
 a=|\Phi_G(t)-\Phi_G(s)|,\qquad a'=|\Phi_{G'}(t)-\Phi_{G'}(s)|,
 \qquad G_{i,u}=(G_i')^-\cap[u,u^+).
\]
By \eqref{eq:chop-variation},
\begin{equation}\label{eq:chop-refine-increments}
 a_{i,u}=\sum_{v\in G_{i,u}}a'_{i,v}.
\end{equation}
For $z'\in C_p^{G'}$ take $z_{i,u}=\max_{v\in G_{i,u}}|z'_{i,v}|$.
We have
\[
 \sum_{i,u}V_{\beta_i}(z_{i,u})\le\sum_{i,v}V_{\beta_i}(z'_{i,v})\le p,
 \qquad
 \sum_{i,v}a'_{i,v}z'_{i,v}\le\sum_{i,u}a_{i,u}z_{i,u}.
\]
Hence $z\in C_p^G$ and $\rho_p^{G'}(a')\le\rho_p^G(a)$.
If \eqref{eq:chop-refine-card} holds, take $z\in C_p^G$ and
$z'_{i,v}=z_{i,u}$ for $v\in G_{i,u}$. Then
\[
 z'\in C_{Kp}^{G'},\qquad \sum_{i,v}a'_{i,v}z'_{i,v}=\sum_{i,u}a_{i,u}z_{i,u}.
\]
Thus $\rho_p^G(a)\le\rho_{Kp}^{G'}(a')\le K\rho_p^{G'}(a')$
by \eqref{eq:est-scaling}. Use $\rho_p^G(\zeta)=\rho_p^G(|\zeta|)$.

For \eqref{eq:chop-refine-F} it suffices to take $A$ finite. Suppose
$G'_{i_0}=G_{i_0}\cup\{v\}$ with $u<v<u^+$ for some $u\in G_{i_0}^-$,
and $G_i'=G_i$ for $i\ne i_0$. Put
\[
 b_t=\sum_{(i,w)\in I_G\setminus\{(i_0,u)\}}
             \varphi_{w,w^+}(t_i)\xi_{i,w}^{(\beta_i)},\qquad
 f_1=\varphi_{u,v},\quad f_2=\varphi_{v,u^+}.
\]
Set
\[
 M(x,y)=\max_{t\in A}\{b_t+x f_1(t_{i_0})+y f_2(t_{i_0})\}.
\]
For $x\le y$, choose maximizers $t,s$ for $M(x,y),M(y,x)$, respectively.
Then
\[
 \begin{aligned}
 &M(x,x)+M(y,y)-M(x,y)-M(y,x)\\
 &\qquad\ge(y-x)
 \begin{cases}
 f_2(s_{i_0})-f_2(t_{i_0}),&t_{i_0}\le s_{i_0},\\
 f_1(t_{i_0})-f_1(s_{i_0}),&s_{i_0}<t_{i_0}
 \end{cases}
 \ \ge0 .
 \end{aligned}
\]
Take independent copies $Z_1,Z_2$ of $\xi_{i_0}^{(\beta_{i_0})}$
independent of $(b_t)_{t\in A}$. By \eqref{eq:chop-sum},
$f_1+f_2=\varphi_{u,u^+}$, so
\[
 \begin{aligned}
 2F_{G'}(A)&=\E\bigl(M(Z_1,Z_2)+M(Z_2,Z_1)\bigr)\\
           &\le\E\bigl(M(Z_1,Z_1)+M(Z_2,Z_2)\bigr)=2F_G(A).
 \end{aligned}
\]
Put $G^{(n)}=(G'_1,\ldots,G'_n,G_{n+1},\ldots)$. Inserting the
points of $G_i'\setminus G_i$ for $i\le n$ gives
$F_{G^{(n)}}(A)\le F_G(A)$. Define $X_t(G')$ and $X_t(G^{(n)})$ on the same probability
space, using the same variables for $i\le n$. Then by \eqref{eq:chop-square}
\[
 \begin{aligned}
 F_{G'}(A)
 &\le F_{G^{(n)}}(A)+\sum_{t\in A}\|X_t(G')-X_t(G^{(n)})\|_1\\
 &\le F_G(A)+2\sqrt2\sum_{t\in A}\|t_{\{i>n\}}\|_2 .
 \end{aligned}
\]
Let $n\to\infty$ and use \eqref{eq:intro-width}.
\end{proof}

\subsection{Basic estimates}

We shall use the following lemma in the construction of the
partitions and in the proof of \eqref{eq:chop-main-decomp}. Recall that
\[
 \rho_p^\beta(v)=
 \sup_{\substack{z\in C_p^\beta\\z_i\ge0\ (i\ge1)}}\sum_i|v_i|z_i .
\]

\begin{lemma}\label{lem:chop-threshold}
Suppose $q\ge1$ and $\rho_q^\beta(v)\le\delta$. Put
\[
 h=\frac{2\delta}{q},\qquad J=\{i:|v_i|>h\}.
\]
Then
\begin{equation}\label{eq:chop-threshold}
 \sum_{i\in J}\beta_i\le\frac q2,\qquad
 \sum_{i\in J}\beta_i|v_i|\le\delta,\qquad
 \|v_{J^c}\|_2\le\frac{2\delta}{\sqrt q}.
\end{equation}
\end{lemma}

\begin{proof}
If $\delta=0$, then $v=0$. Suppose $\delta>0$, so that $J$
is finite. If $\sum_{i\in J}\beta_i>q/2$, we can take $z$ with
\[
 0\le z_i\le\beta_i,\qquad z=z_{J},\qquad \sum_i z_i=q/2 .
\]
Since $V_{\beta_i}(z_i)\le2z_i$, we have $z\in C_q^\beta$ and
\[
 \rho_q^\beta(v)\ge\sum_{i\in J}|v_i|z_i>hq/2=\delta.
\]
Thus $\sum_{i\in J}\beta_i\le q/2$. Taking
$z_i=\beta_i\ind_{\{i\in J\}}$ gives $z\in C_q^\beta$ and
$\sum_{i\in J}\beta_i|v_i|\le\delta$.

Put $w=v_{J^c}$. If $\|w\|_2\ge h\sqrt q$, set
$z=\sqrt q\,w/\|w\|_2$. Then $\|z\|_\infty\le1$ and
\[
 \sum_iV_{\beta_i}(z_i)=\sum_i z_i^2=q,\qquad
 \delta\ge\rho_q^\beta(v)\ge\langle v,z\rangle=\sqrt q\,\|w\|_2 .
\]
In both cases $\|w\|_2\le h\sqrt q$.
\end{proof}

For $Q=[a,b]$ and $x\in\R$ put
\begin{equation}\label{eq:chop-PQ}
 P_Qx=\min\{b,\max\{x,a\}\}.
\end{equation}
Let $Q=\prod_i Q_i$, where the $Q_i$ are closed bounded intervals
and $Q\cap\ell_2\ne\emptyset$. Put $(P_Qx)_i=P_{Q_i}x_i$ for $x\in\ell_2$.
We have
\[
 |(P_Qx)_i-(P_Qy)_i|\le|x_i-y_i|,\qquad
 \|P_Qx-u\|_2\le\|x-u\|_2\quad(u\in Q\cap\ell_2).
\]
Thus $P_Qx\in\ell_2$. By the first inequality,
\begin{equation}\label{eq:chop-PQ-rho}
 \rho_p^\beta(P_Qx-P_Qy)\le\rho_p^\beta(x-y).
\end{equation}

We use Lemma~\ref{lem:chop-threshold} to bound $\|t-P_Qt\|_{1,\beta}$
and the diameter of $P_Q(B)$.

\begin{lemma}\label{lem:chop-clip}
Let $u\in B\subset\ell_2$, $q\ge1$ and $\delta>0$. Suppose that
\begin{equation}\label{eq:chop-qsmall}
 \sup_{t\in B}\rho_q^\beta(t-u)\le\delta .
\end{equation}
Put $h=2\delta/q$ and $k_i=\lfloor u_i/h\rfloor$. Define
\[
 Q_i=[(k_i-1)h,(k_i+2)h],\qquad Q=\prod_i Q_i,\qquad y(t)=P_Qt .
\]
Then
\begin{align}
 \sup_{t\in B}\|t-y(t)\|_{1,\beta}
 &\le\delta, \label{eq:chop-clip-l1}\\
 \sup_{t\in B}\|y(t)-u\|_2
 &\le\sqrt{12}\,\frac{\delta}{\sqrt q},\qquad
 \sup_{t\in B}\|y(t)-u\|_\infty\le\frac{4\delta}{q},
 \label{eq:chop-clip-2inf}\\
 \Delta(y(B),\rho_p^\beta)
 &\le16\delta\sqrt{\frac pq}\quad(1\le p\le q).
 \label{eq:chop-clip-rho}
\end{align}
For closed intervals $I_i\supset\{t_i:t\in B\}$ we may replace
$Q_i$ by $Q_i\cap I_i$. If $G_i$ is finite and
$\min G_i\le y_i(t)\le\max G_i$ for all $t\in B$, then
\[
 \Delta(y(B),d_{p,G})\le16\delta\sqrt{p/q}\quad(1\le p\le q).
\]
\end{lemma}

\begin{proof}
For $t\in B$ put $v=t-u$, $J=\{i:|v_i|>h\}$. We have
\[
 y_i(t)=t_i\quad(i\notin J),\qquad
 |y_i(t)-u_i|\le2h,\qquad |t_i-y_i(t)|\le|v_i|.
\]
By \eqref{eq:chop-threshold},
\[
 \begin{aligned}
 \|t-y(t)\|_{1,\beta}&\le\sum_{i\in J}\beta_i|v_i|\le\delta,\\
 \|y(t)-u\|_2^2&\le \|v_{J^c}\|_2^2+4h^2|J|
 \le\frac{4\delta^2}{q}+4h^2\frac q2=\frac{12\delta^2}{q}.
 \end{aligned}
\]
Also $\|y(t)-u\|_\infty\le2h$. For $1\le p\le q$, \eqref{eq:est-rho-infty} gives
\[
 \rho_p^\beta(y(t)-u)
 \le \sqrt p\,\|y(t)-u\|_2+p\|y(t)-u\|_\infty
 \le8\delta\sqrt{p/q}.
\]
This gives \eqref{eq:chop-clip-rho}. By \eqref{eq:chop-refine-rho} for
$\{\min G_i,\max G_i\}\subset G_i$,
\[
 d_{p,G}(y(s),y(t))\le\rho_p^\beta(y(s)-y(t))\qquad(s,t\in B).
\]
If $\{t_i:t\in B\}\subset I_i$, then
\[
 P_{Q_i}(t_i)\in[t_i\wedge u_i,t_i\vee u_i]\subset I_i,\qquad
 P_{Q_i\cap I_i}(t_i)=P_{Q_i}(t_i).
\]
\end{proof}

We will also use a bound for $\rho_p^\beta(s-t)$ in terms of $d_{p,G}(s,t)$.

\begin{lemma}\label{lem:chop-inverse}
Let $h>0$, $u_i,v_i\in h\Z$ with $u_i\le v_i$, and put
$G_i=[u_i,v_i]\cap h\Z$.
Suppose $s,t\in\ell_2$ with $s_i,t_i\in[u_i,v_i]$ for all $i$. For $p\ge1$,
\begin{equation}\label{eq:chop-forward}
 d_{p,G}(s,t)\le\rho_p^\beta(t-s).
\end{equation}
If $\eta=d_{p,G}(s,t)<ph/2$, then
\begin{equation}\label{eq:chop-inverse}
 \rho_p^\beta(t-s)\le8\eta.
\end{equation}
\end{lemma}

\begin{proof}
For $G_i^0=\{u_i,v_i\}$ we have $d_{p,G^0}(s,t)=\rho_p^\beta(t-s)$.
Thus \eqref{eq:chop-forward} follows from \eqref{eq:chop-refine-rho}.
For $u\in G_i^-$ put
\[
 e_{i,u}=|\varphi_{u,u+h}(t_i)-\varphi_{u,u+h}(s_i)|
 =|[u,u+h]\cap[s_i\wedge t_i,s_i\vee t_i]|.
\]
We have
\[
 \begin{gathered}
 \sum_{u\in G_i^-}e_{i,u}=|t_i-s_i|,\\
 \#\{u\in G_i^-:0<e_{i,u}<h\}
 \le\#\{u\in G_i^-:(u,u+h)\cap\{s_i,t_i\}\ne\emptyset\}\le2 .
 \end{gathered}
\]
Put $J=\{(i,u)\in I_G:e_{i,u}>2\eta/p\}$.
Since $\rho_p^G(e)=\eta$, \eqref{eq:chop-threshold} gives
\[
 \sum_{(i,u)\in J}\beta_i e_{i,u}\le\eta,\qquad
 \sum_{(i,u)\in I_G\setminus J}e_{i,u}^2\le\frac{4\eta^2}{p}.
\]
Put
\[
 w_i^{(1)}=\sum_{u:(i,u)\in J}e_{i,u},\qquad
 w_i^{(2)}=\sum_{u:(i,u)\in I_G\setminus J}e_{i,u}.
\]
Since $2\eta/p<h$, each sum defining $w_i^{(2)}$ has at most two
nonzero terms. Thus
\[
 |t-s|=w^{(1)}+w^{(2)},\qquad
 \|w^{(1)}\|_{1,\beta}\le\eta,\qquad
 \|w^{(2)}\|_2\le\frac{2\sqrt2\,\eta}{\sqrt p},\qquad
 \|w^{(2)}\|_\infty\le\frac{4\eta}{p}.
\]
By \eqref{eq:est-l1} and \eqref{eq:est-rho-infty},
\[
 \rho_p^\beta(t-s)=\rho_p^\beta(|t-s|)
 \le\eta+\sqrt p\,\|w^{(2)}\|_2+p\|w^{(2)}\|_\infty\le8\eta .
\]
\end{proof}

Apply Lemma~\ref{lem:est-growth} to the vectors $\Phi_G(t)$, with
cap $\beta_i$ at $(i,u)$. This gives \eqref{eq:est-growth} with
$F_G$ and $d_{p,G}$ in place of $S_\beta$ and $\rho_p^\beta$.

For $t\in\ell_2$ and $\sigma>0$, write
\[
 B_{p,G}(t,\sigma)=\{s\in\ell_2:d_{p,G}(s,t)\le\sigma\}.
\]
We use the growth condition with $G$ fixed to construct a partition
of $A$. See \cite[Corollary~2.9]{BL}.

\begin{lemma}\label{lem:chop-greedy}
There is a universal constant $L_0\ge1$ such that the following
holds. Let $m\ge8$, $p=\log m$, $v>0$, and let $G=(G_i)_i$ be as above.
For every nonempty $A\subset\ell_2$, there is a partition
\[
 A=A_1\cup\cdots\cup A_s,\qquad s\le m,
\]
such that $\Delta(A_\ell,d_{p,G})\le L_0v$ for $\ell<s$.
For $A_s$ either $\Delta(A_s,d_{p,G})\le L_0v$ or
\begin{equation}\label{eq:chop-greedy-drop}
 \emptyset\ne D\subset A_s,\qquad
 \Delta(D,d_{p,G})\le v
 \quad\Longrightarrow\quad
 F_G(D)\le F_G(A)-v.
\end{equation}
\end{lemma}

\begin{proof}
If $F_G(A)=\infty$, take $A_1=A$ in \eqref{eq:chop-greedy-drop}.
Let $c_1,C_1>0$ be the constants in the growth condition for $X_t(G)$.
Choose $L\ge1$ so that $c_1L-C_1\ge2$, and put $L_0=2L$.
Put $U_1=A$. For $j<m$ and $U_j\ne\emptyset$, choose $t_j\in U_j$ with
\[
 F_G\bigl(A\cap B_{p,G}(t_j,v)\bigr)
 \ge\sup_{t\in U_j}F_G\bigl(A\cap B_{p,G}(t,v)\bigr)-v/2,
\]
and put
\[
 A_j=U_j\cap B_{p,G}(t_j,Lv),\qquad U_{j+1}=U_j\setminus A_j .
\]
$\Delta(A_j,d_{p,G})\le L_0v$. If the $A_j$ cover $A$, take these parts.
Otherwise put $A_m=U_m$. For nonempty $D\subset A_m$ with
$\Delta(D,d_{p,G})\le v$, take $t_m\in D$. For $j<k\le m$,
$t_k\in U_{j+1}$, so
$d_{p,G}(t_j,t_k)>Lv$. Also $t_m\in U_j$ for $j<m$ and
$D\subset A\cap B_{p,G}(t_m,v)$. Hence
\[
 F_G\bigl(A\cap B_{p,G}(t_j,v)\bigr)
 \ge F_G(D)-v/2,
 \qquad j\le m.
\]
The growth condition, applied to the $m$ sets
$A\cap B_{p,G}(t_j,v)$, gives
\[
 F_G(A)\ge F_G(D)-v/2+c_1Lv-C_1v\ge F_G(D)+v.
\]
\end{proof}

\subsection{A decomposition lemma}

We combine Lemma~\ref{lem:chop-greedy} with Lata\l a's principle.
See \cite[Lemma~10.7.3]{TalagrandBook}.

Let $A$ be nonempty, $f:A\to\ell_2$ and $H>0$. For each $i$
let $Q_i$ be a closed interval containing $f_i(A)$, with
$\min Q_i,\max Q_i\in H\Z$.
For $h>0$ with $H/h$ a positive integer, put
\begin{equation}\label{eq:part-cutting}
 G_h^A(i)=Q_i\cap h\Z,\qquad G_h^A=(G_h^A(i))_{i\ge1}.
\end{equation}
For $\emptyset\ne B\subset A$ and $s,t\in A$, put
\begin{equation}\label{eq:part-local-notation}
 F_h^A(B)=F_{G_h^A}(f(B)),\qquad
 d_{p,h}^A(s,t)=d_{p,G_h^A}(f(s),f(t)).
\end{equation}
Since $G_H^A(i)\subset G_h^A(i)$, Lemma~\ref{lem:chop-refine} gives
\begin{equation}\label{eq:part-refine-F}
 F_h^A(B)\le F_H^A(B)\qquad(\emptyset\ne B\subset A).
\end{equation}
Let $c_0,C_0$ be the constants in \eqref{eq:est-latala-drop}, with
$c_0\le1$, and put $d_0=c_0/20$. Let $L_0$ be as in
Lemma~\ref{lem:chop-greedy}.

\begin{lemma}\label{lem:chop-three}
Let $R\ge4$ and $m\ge8$ be integers. Put
\[
 h=H/R,\qquad q=\log m,\qquad v=\frac{qh}{16L_0},\qquad
 \delta=\frac{qh}{2}.
\]
Let $m_*=0$ or $8\le m_*\le m$, with $m_*$ an integer. If $m_*>0$,
put $p_*=\log m_*$ and suppose there are $M$ and $b>0$ such that
\begin{align}
 \emptyset\ne D\subset A,\quad \Delta(D,d_{p_*,H}^A)\le b
 &\quad\Longrightarrow\quad F_H^A(D)\le M-b,\label{eq:part-old-ineq}\\
 16\delta\sqrt{p_*/q}&\le\frac{c_0 b}{20C_0}.\label{eq:part-old-ratio}
\end{align}
There is a partition $\mathcal B$ of $A$ with $|\mathcal B|\le m(m_*+1)$
such that each $B\in\mathcal B$ satisfies one of the following.

\medskip
\noindent\textup{\textbf{(a)}} For $\emptyset\ne D\subset B$,
\begin{equation}\label{eq:part-remaining}
 \Delta(D,d_{q,h}^A)\le v
 \quad\Longrightarrow\quad F_h^A(D)\le F_h^A(A)-v.
\end{equation}

\medskip
\noindent\textup{\textbf{(b)}} We have
\begin{equation}\label{eq:part-small-original}
 \Delta(B,d_{q,h}^A)\le L_0v,\qquad \Delta(f(B),\rho_q^\beta)\le\delta.
\end{equation}
If $m_*>0$, we also have
\begin{equation}\label{eq:part-old-covered}
 \Delta(B,d_{p_*,H}^A)\le b,\qquad F_h^A(B)\le M-b.
\end{equation}

\medskip
\noindent\textup{\textbf{(c)}} We have $m_*>0$, and there are intervals
$Q_i'\subset Q_i$ with $\min Q_i',\max Q_i'\in h\Z$ and
$|Q_i'|\le3h$. Put $Q'=\prod_i Q_i'$, $y(t)=P_{Q'}f(t)$ and
$G'=(Q_i'\cap h\Z)_{i\ge1}$. Then
\begin{gather}
 \Delta(f(B),\rho_q^\beta)\le\delta,\qquad
 \sup_{t\in B}\|f(t)-y(t)\|_{1,\beta}\le\delta,\label{eq:part-old-error}\\
 \Delta(y(B),d_{p_*,G'})\le\frac{c_0b}{20C_0},\qquad
 F_{G'}(y(B))\le F_H^A(A)-d_0b.\label{eq:part-old-newdrop}
\end{gather}
\end{lemma}

\begin{proof}
Apply Lemma~\ref{lem:chop-greedy} to $f(A)$ with $G=G_h^A$, $p=q$
and $v$ as above. If the last set satisfies \eqref{eq:chop-greedy-drop},
we take its inverse image under $f$ as a set satisfying
\textup{\textbf{(a)}}. Each of the other inverse images $U$ satisfies
$\Delta(U,d_{q,h}^A)\le L_0v$ and hence
\[
 L_0v=qh/16<qh/2,\qquad
 \Delta(f(U),\rho_q^\beta)\le8L_0v=\delta
\]
by Lemma~\ref{lem:chop-inverse}. If $m_*=0$, take $B=U$ in
\textup{\textbf{(b)}}. Suppose $m_*>0$.

Take $u\in f(U)$ and $k_i=\lfloor u_i/h\rfloor$. Put
\[
 Q_i'=Q_i\cap[(k_i-1)h,(k_i+2)h],\qquad
 y(t)=P_{\prod_i Q_i'}f(t),\qquad G_i'=Q_i'\cap h\Z.
\]
We have $u_i\in Q_i'$ for all $i$, and $u\in\ell_2$. Since
$H\Z\subset h\Z$, the endpoints of $Q_i'$ belong to $h\Z$.
By Lemma~\ref{lem:chop-clip} and \eqref{eq:part-old-ratio},
\[
 \sup_{t\in U}\|f(t)-y(t)\|_{1,\beta}\le\delta,\qquad
 \Delta(y(U),d_{p_*,G'})\le16\delta\sqrt{p_*/q}\le\frac{c_0 b}{20C_0}.
\]
Put
\[
 G_i^*=G_H^A(i)\cup G_i',\qquad
 G^*=(G_i^*)_{i\ge1},\qquad J=I_{G'},\qquad
 \Phi=\Phi_{G^*}\circ f.
\]
For consecutive points $w<w^+$ of $G_H^A(i)$,
\[
 \#(G_i^*\cap[w,w^+))\le1+|G_i'|\le5 .
\]
Hence Lemma~\ref{lem:chop-refine} gives
\[
 d_{p_*,H}^A(s,t)\le5d_{p_*,G^*}(f(s),f(t)),\qquad
 F_{G^*}(f(U))\le F_H^A(U).
\]
We have $G_i^*\cap Q_i'=G_i'$ and $J\subset I_{G^*}$. For $(i,w)\in J$,
if $w<w^+$ are consecutive in $G_i'$, they are consecutive in $G_i^*$.
For $t\in U$,
\[
 (\Phi(t))_{i,w}
 =\varphi_{w,w^+}(f_i(t))
 =\varphi_{w,w^+}(y_i(t))
 =(\Phi_{G'}(y(t)))_{i,w}\qquad((i,w)\in J).
\]
Use \eqref{eq:est-XI} and \eqref{eq:est-rhoI} on $I_{G^*}$ with
$\beta^*_{i,w}=\beta_i$. Thus
\[
 \Delta(\Phi(U),\rho_{p_*,J}^{\beta^*})=\Delta(y(U),d_{p_*,G'})
 \le\frac{c_0 b}{20C_0},\qquad
 S_J(\Phi(B))=F_{G'}(y(B))\quad(\emptyset\ne B\subset U).
\]

Apply Proposition~\ref{prop:est-latala-principle} to $\Phi(U)$ with
$I=I_{G^*}$, $J$ as above, and $\sigma=b/10$. For each member $D$ of this partition put
$B=\{t\in U:\Phi(t)\in D\}$.
For the at most $m_*$ sets given by \eqref{eq:est-latala-cover},
\[
 \Delta(B,d_{p_*,H}^A)\le5\Delta(D,\rho_{p_*}^{G^*})\le b,\qquad
 F_h^A(B)\le F_H^A(B)\le M-b .
\]
These sets satisfy \textup{\textbf{(b)}}. For the set from
\eqref{eq:est-latala-drop}, if present,
\[
 F_{G'}(y(B))
 \le F_{G^*}(f(U))-\frac{c_0 b}{10}+C_0\frac{c_0 b}{20C_0}
 \le F_H^A(A)-d_0 b .
\]
This gives \textup{\textbf{(c)}}. There are at most $m(m_*+1)$ sets.
\end{proof}

\subsection{Construction of the partitions}\label{sec:partitions}

When $B$ satisfies \textup{\textbf{(a)}}, we take $H(B)=H(A)$.
We shall use \eqref{eq:part-remaining} at a later step, when the
length of the chopping intervals is $H(A)/R$.

For each $(n,A)$ with $A\in\cA_n$, we choose $f_A:A\to\ell_2$,
$H(A)>0$, intervals $Q_i(A)$ and an integer $m_*(A)\ge0$.
All these data refer to the pair $(n,A)$, even if the same set
occurs in several partitions. In $F_h^A$ and $d^A_{p,h}$ take $f=f_A$ and $G_h^A(i)=Q_i(A)\cap h\Z$.

Choose a power of two $R\ge4$ such that
\begin{equation}\label{eq:part-R}
 \frac{256L_0}{\sqrt R}\le\frac{c_0}{20C_0}.
\end{equation}
Put $\kappa=d_0/(64L_0R)$ and
\begin{equation}\label{eq:part-parameters}
 m_n=8^{2^n},\qquad p_n=\log m_n=(\log8)2^n.
\end{equation}
We require $f_{A,i}(A)\subset Q_i(A)$ and
$\min Q_i(A),\max Q_i(A)\in H(A)\Z$.
Write
\begin{equation}\label{eq:part-functional-charge}
 F(A)=F_{H(A)}^A(A),\qquad
  \tau(A)=p_n H(A),\qquad
 \eta(A)=\frac{\tau(A)}{16L_0R}.
\end{equation}

If $T=\{t\}$, take $\cA_k=\{T\}$, $z(t)=u_k(T)=t$ for all
$k$, and $a(t)=v(t)=0$. Suppose $|T|>1$. Fix $t_0\in T$ and write $T$ for the translated set $T-t_0$.
The expected supremum is unchanged by translation. Put $D=\Delta(T,d_2)$ and
\begin{equation}\label{eq:part-root}
 \cA_0=\{T\},\qquad
 f_T(t)=t,\qquad H(T)=D,\qquad Q_i(T)=[-D,D].
\end{equation}
Since $|t_i|\le D$ on $T$, we apply Lemma~\ref{lem:chop-refine}
to $\{-D,D\}\subset G_D^T(i)=\{-D,0,D\}$ and obtain
\begin{equation}\label{eq:part-root-F}
 F(T)\le S_\beta(T),\qquad \tau(T)=(\log8)D.
\end{equation}

Suppose $\cA_n$ is constructed and $A\in\cA_n$. Let
$A_j\in\cA_j$ contain $A$ for $0\le j\le n$. Put
$H=H(A)$, $h=H/R$, $q=p_n$ and $\delta=qh/2$.

We use \eqref{eq:part-remaining} to verify \eqref{eq:part-old-ineq}.
Take $m_*(A)=0$ unless there is an index $1\le s<n$ such that
\begin{equation}\label{eq:part-old-scale}
 \begin{gathered}
 A_s\text{ is given by \textup{\textbf{(a)}} in the partition of }A_{s-1},\\
 H(A_j)=H(A_s)/R,\quad Q_i(A_j)=Q_i(A_s)\quad(s<j\le n),\\
 f_{A_j}(t)=f_{A_s}(t)\quad(s<j\le n,\ t\in A_j),\\
 \tau(A)\le2\tau(A_s).
 \end{gathered}
\end{equation}
There is at most one such $s$: if $s<s'<n$ both satisfy
\eqref{eq:part-old-scale}, then $H(A_n)=H(A_s)/R=H(A_{s'})$ and
$H(A_n)=H(A_{s'})/R$, a contradiction. When such an $s$ exists, put
\[
 \begin{gathered}
 m_*(A)=m_{s-1},\qquad p_*=p_{s-1},\qquad M=F_H^{A_{s-1}}(A_{s-1}),\\
 b=\eta(A_{s-1})=\frac{\tau(A_s)}{32L_0R}.
 \end{gathered}
\]
We have $A\subset A_s$ and
\[
 G_H^A=G_{H(A_{s-1})/R}^{A_{s-1}},\qquad
 f_A(t)=f_{A_{s-1}}(t)\quad(t\in A).
\]
Hence \eqref{eq:part-remaining} for $A_s$ gives \eqref{eq:part-old-ineq}
for $A$. Since $q/p_*=2R\tau(A)/\tau(A_s)$,
\[
 \frac{16\delta}{b}\sqrt{\frac{p_*}{q}}
 =\frac{256L_0}{\sqrt{2R}}\sqrt{\frac{\tau(A)}{\tau(A_s)}}
 \le\frac{256L_0}{\sqrt R}\le\frac{c_0}{20C_0}.
\]
This is \eqref{eq:part-old-ratio}.

Apply Lemma~\ref{lem:chop-three} with $m=m_n$ and $m_*=m_*(A)$.
Include each member $B$ of this partition in $\cA_{n+1}$.
For each $B$ choose one of the conclusions \textup{\textbf{(a)}},
\textup{\textbf{(b)}} and \textup{\textbf{(c)}} which it satisfies. Put
\[
 (H(B),f_B(t),Q_i(B))=
 \begin{cases}
 (H,f_A(t),Q_i(A)),&\text{in \textup{\textbf{(a)}}},\\
 (h,f_A(t),Q_i(A)),&\text{in \textup{\textbf{(b)}}},\\
 (h,y(t),Q_i'),&\text{in \textup{\textbf{(c)}}}
 \end{cases}\qquad(t\in B).
\]
In \textup{\textbf{(b)}} with $m_*(A)>0$, \eqref{eq:part-old-covered} gives
\begin{equation}\label{eq:part-delayed-drop}
 F(A_{s-1})-F(B)\ge b
\end{equation}
since $M\le F(A_{s-1})$. In \textup{\textbf{(c)}}, \eqref{eq:part-old-newdrop} gives
\begin{equation}\label{eq:part-immediate-drop}
 F(A)-F(B)\ge d_0b=2\kappa\tau(A_s)\ge\kappa\tau(A).
\end{equation}
By \eqref{eq:part-old-error} in case \textup{\textbf{(c)}} and $f_B(t)=f_A(t)$
for $t\in B$ in cases \textup{\textbf{(a)}} and \textup{\textbf{(b)}},
\begin{equation}\label{eq:part-clipped-error}
 \sup_{t\in B}\|f_A(t)-f_B(t)\|_{1,\beta}\le\frac{\tau(A)}{2R}.
\end{equation}

In all three cases, by \eqref{eq:part-refine-F} and
\eqref{eq:part-immediate-drop},
\begin{equation}\label{eq:part-scale-change}
 H(B)\in\{H(A),H(A)/R\},\qquad F(B)\le F(A).
\end{equation}
The number of parts of $A$ is at most $m_n(m_*(A)+1)\le2m_n^2$.
Thus
\begin{equation}\label{eq:part-cardinality}
 \log_2|\cA_n|
 \le\sum_{j<n}(1+6\cdot2^j)\le7\cdot2^n,
 \qquad |\cA_n|\le2^{2^{n+3}}.
\end{equation}

\subsection{\texorpdfstring{A bound for $\sum_{n\ge0}2^nH(\cA_n(t))$}{The sum of the chopping lengths}}

The following lemma is due to Talagrand \cite[Lemma~2.9.5]{TalagrandBook}.
See also \cite[Lemma~6.1]{BL}.

\begin{lemma}\label{lem:part-maximal}
Let $(b_n)$ be a bounded sequence of positive numbers and let $\vartheta>1$. Put
\[
 V=\{n:b_m<b_n\vartheta^{|m-n|}\text{ for every }m\ne n\}.
\]
Then
\begin{equation}\label{eq:part-maximal}
 \sum_nb_n\le\frac{2\vartheta}{\vartheta-1}\sum_{n\in V}b_n.
\end{equation}
\end{lemma}

\begin{proof}
Put $n_0=n$. If $n_j\notin V$, choose $n_{j+1}\ne n_j$ with
$b_{n_{j+1}}\ge b_{n_j}\vartheta^{|n_{j+1}-n_j|}$. Thus
\[
 b_{n_j}\ge b_n\vartheta^{\sum_{i<j}|n_{i+1}-n_i|}
 \ge b_n\vartheta^{|n_j-n|},
 \qquad b_{n_j}\ge b_n\vartheta^j.
\]
Since $(b_n)$ is bounded, this ends at some $m\in V$ with
$b_n\le b_m\vartheta^{-|m-n|}$. Therefore
\[
 \sum_n b_n
 \le \sum_{m\in V}b_m\sum_n\vartheta^{-|m-n|}
 \le \frac{2\vartheta}{\vartheta-1}\sum_{m\in V}b_m.
\]
\end{proof}

Fix $t\in T$ and put $A_n=\cA_n(t)$. Write
\begin{equation}\label{eq:part-branch-notation}
 H_n=H(A_n),
 \qquad
 F_n=F(A_n),
 \qquad
 \tau_n=p_nH_n.
\end{equation}
By \eqref{eq:part-scale-change},
\begin{equation}\label{eq:part-charge-recursion}
 \tau_{n+1}=
 \begin{cases}
 2\tau_n,& H_{n+1}=H_n,\\
 2\tau_n/R,& H_{n+1}=H_n/R.
 \end{cases}
\end{equation}
If $H_{n+1}=H_n$, use $D=\{t\}$ in \eqref{eq:part-remaining} and
\eqref{eq:part-refine-F} to obtain
\[
 F_n\ge F_{H_n/R}^{A_n}(A_n)
 \ge F_{H_n/R}^{A_n}(\{t\})+\eta(A_n)
 =\frac{\tau_n}{16L_0R}.
\]
Hence
\begin{equation}\label{eq:part-charge-bounded}
 \sup_n\tau_n\le\max\{\tau_0,32L_0RF_0\}.
\end{equation}

By \eqref{eq:part-charge-recursion} and \eqref{eq:part-charge-bounded},
$H_{n+1}=H_n/R$ for infinitely many $n$. Put
\[
 I=\{n:H_{n+1}=H_n/R\}=\{e_0<e_1<\cdots\}.
\]
Apply Lemma~\ref{lem:part-maximal} to $(\tau_n)$ with $\vartheta=2$
and $V=\{n:\tau_m<2^{|m-n|}\tau_n\text{ for all }m\ne n\}$. If $n\notin I$, then
$\tau_{n+1}=2\tau_n$, and hence $n\notin V$. Thus $V\subset I$, and
\begin{equation}\label{eq:part-first-selection}
 \sum_n\tau_n\le4\sum_j\tau_{e_j}.
\end{equation}
To use \eqref{eq:part-old-scale} at the next decrease of $H$, we
need $\tau_{e_{j+1}}\le2\tau_{e_j}$. Put
$W=\{j:\tau_{e_i}<2^{|i-j|}\tau_{e_j}\text{ for all }i\ne j\}$.
Apply the lemma to $(\tau_{e_j})$ to obtain
\begin{equation}\label{eq:part-second-selection}
 \sum_n\tau_n\le16\sum_{j\in W}\tau_{e_j}.
\end{equation}

Fix $j\in W$ with $s=e_j>0$. If $s-1\in I$, then $j\ge1$
and $s-1=e_{j-1}$, so
\[
 2\tau_s\le(R/2)\tau_s=\tau_{e_{j-1}}<2\tau_s.
\]
Thus $s-1\notin I$ and $A_s$ is given by \textup{\textbf{(a)}} in the partition of $A_{s-1}$.

If $A_{s+1}$ is given by \textup{\textbf{(c)}} in the partition of $A_s$, then
\begin{equation}\label{eq:part-drop-immediate}
 F_s-F_{s+1}\ge\kappa\tau_s.
\end{equation}
If $A_{s+1}$ is given by \textup{\textbf{(b)}} in the partition of $A_s$ with $m_*(A_s)>0$, the index
in \eqref{eq:part-old-scale} for $A_s$ is $\ell=e_{j-1}$. By
\eqref{eq:part-delayed-drop} and $\tau_s\le2\tau_\ell$,
\[
 F_{\ell-1}-F_{s+1}\ge\frac{\tau_\ell}{32L_0R}\ge\kappa\tau_s.
\]

Suppose $A_{s+1}$ is given by \textup{\textbf{(b)}} in the partition of $A_s$ with $m_*(A_s)=0$. We use
the estimate on small subsets of $A_s$ at the next decrease of $H$.
Let $n=e_{j+1}$. We have $\tau_n<2\tau_s$ and
\[
 \begin{gathered}
 H(A_k)=H(A_s)/R,\qquad Q_i(A_k)=Q_i(A_s)\quad(s<k\le n),\\
 f_{A_k}(t)=f_{A_s}(t)\quad(s<k\le n,\ t\in A_k).
 \end{gathered}
\]
Thus \eqref{eq:part-old-scale} holds at $A_n$ with index $s$.
By \eqref{eq:part-delayed-drop} in case \textup{\textbf{(b)}} and
\eqref{eq:part-immediate-drop} in case \textup{\textbf{(c)}},
\begin{equation}\label{eq:part-drop-later}
 F_{s-1}-F_{n+1}\ge\kappa\tau_s .
\end{equation}

In the three cases considered above, put, respectively,
\[
 J_j=\{s\},\qquad J_j=[e_{j-1}-1,s]\cap\Z,\qquad
 J_j=[s-1,e_{j+1}]\cap\Z.
\]
Thus $\kappa\tau_s\le\sum_{m\in J_j}(F_m-F_{m+1})$. For $m\in J_j$,
\[
 j\in
 \begin{cases}
 \{0,1\},&m<e_0,\\
 \{k-1,k,k+1,k+2\},&e_k\le m<e_{k+1}.
 \end{cases}
\]
Hence $\kappa\sum_{j\in W:e_j>0}\tau_{e_j}
\le4\sum_{m\ge0}(F_m-F_{m+1})\le4F_0$. By \eqref{eq:part-second-selection},
\begin{equation}\label{eq:part-charge-sum}
 \sum_{n\ge0}\tau_n\le16\tau_0+64\kappa^{-1}F_0 .
\end{equation}

\subsection{Proof of Theorem~\ref{thm:chop-main}}

We construct $z(t)$ from the maps $f_{\cA_n(t)}$.
Fix $t\in T$ and put $A_n=\cA_n(t)$ and $F_n=F(A_n)$.
In cases \textup{\textbf{(a)}} and \textup{\textbf{(b)}} we have $f_{A_{n+1}}(t)=f_{A_n}(t)$.
In case \textup{\textbf{(c)}} use \eqref{eq:part-immediate-drop} and
\eqref{eq:part-clipped-error}. Thus
\[
 \|f_{A_{n+1}}(t)-f_{A_n}(t)\|_{1,\beta}
 \le\frac{F_n-F_{n+1}}{2R\kappa}.
\]
Put
\[
 z(t)=t+\sum_{n\ge0}\bigl(f_{A_{n+1}}(t)-f_{A_n}(t)\bigr)
 \quad\text{in }\ell_2 .
\]
The sum is absolutely convergent, and
\begin{equation}\label{eq:part-reserve}
 \|t-z(t)\|_{1,\beta}
 \le\sum_{n\ge0}\|f_{A_{n+1}}(t)-f_{A_n}(t)\|_{1,\beta}
 \le\frac{F(T)}{2R\kappa}.
\end{equation}

We now choose vectors $\pi_n(A)$, using the diameter bound in
\textup{\textbf{(b)}} and \textup{\textbf{(c)}}. Put $\pi_0(T)=0$. For $B\in\cA_{n+1}$, $B\subset A\in\cA_n$,
choose
\[
 \pi_{n+1}(B)\in
 \begin{cases}
 \{\pi_n(A)\},&H(B)=H(A),\\
 f_B(B),&H(B)=H(A)/R.
 \end{cases}
\]

Fix $t\in T$ and put $A_n=\cA_n(t)$, $H_n=H(A_n)$ and
$\tau_n=p_nH_n$. We shall prove that
\begin{equation}\label{eq:part-rep-bound}
 \rho_{p_n}^\beta\bigl(f_{A_n}(t)-\pi_n(A_n)\bigr)\le2\tau_n.
\end{equation}
If $H_n=H_0$, then $f_{A_n}(t)=t$ and $\pi_n(A_n)=0$. By
\eqref{eq:est-rho-infty},
\[
 \rho_{p_n}^\beta(t)
 \le\sqrt{p_n}\|t\|_2+p_n\|t\|_\infty\le2\tau_n .
\]

Otherwise let $s<n$ be the last index with $H_{s+1}=H_s/R$.
By \eqref{eq:part-small-original} in case \textup{\textbf{(b)}} and
\eqref{eq:part-old-error} in case \textup{\textbf{(c)}},
\[
 \Delta(f_{A_{s+1}}(A_{s+1}),\rho_{p_s}^\beta)
 \le\Delta(f_{A_s}(A_{s+1}),\rho_{p_s}^\beta)
 \le\frac{\tau_s}{2R},
\]
where the first inequality is \eqref{eq:chop-PQ-rho}.
For $s+1\le j\le n$,
$f_{A_j}(t)=f_{A_{s+1}}(t)$ and $\pi_j(A_j)=\pi_{s+1}(A_{s+1})$.
As $\pi_{s+1}(A_{s+1})\in f_{A_{s+1}}(A_{s+1})$,
\eqref{eq:est-scaling} gives
\[
 \rho_{p_n}^\beta\bigl(f_{A_n}(t)-\pi_n(A_n)\bigr)
 \le\frac{p_n}{p_s}\frac{\tau_s}{2R}
 =\frac{\tau_n}{2}.
\]

Let $B\in\cA_{n+1}$ be contained in $A\in\cA_n$.
For $H(B)=H(A)/R$ take $t'\in B$ with $\pi_{n+1}(B)=f_B(t')$.
By \eqref{eq:part-clipped-error}, \eqref{eq:est-l1} and
\eqref{eq:part-rep-bound},
\begin{equation}\label{eq:part-point-increment}
 \rho_{p_n}^\beta(\pi_{n+1}(B)-\pi_n(A))
 \le\frac{\tau(A)}{2R}+2\tau(A)\le3\tau(A).
\end{equation}
This inequality also holds for $H(B)=H(A)$.

Since $\|h\|_2\le\rho^\beta_p(h)$ for $p\ge1$,
\eqref{eq:part-rep-bound} gives
\[
 \|z(t)-\pi_n(A_n)\|_2
 \le 2\tau_n+\sum_{j\ge n}\|f_{A_{j+1}}(t)-f_{A_j}(t)\|_2
 \xrightarrow[n\to\infty]{}0
\]
by \eqref{eq:part-charge-sum} and \eqref{eq:part-reserve}.

Put $\mathcal B_k=\{T\}$ for $k<3$ and
$\mathcal B_{n+3}=\cA_n$. By \eqref{eq:part-cardinality},
$(\mathcal B_k)$ is admissible. Take $u_k(T)=0$ for $k<3$ and
$u_{n+3}(A)=\pi_n(A)$ for $A\in\cA_n$.
With $A_n=\cA_n(t)$ and $2^{n+4}=16p_n/\log8$,
\eqref{eq:est-scaling} and \eqref{eq:part-point-increment} give
\begin{equation}\label{eq:part-chain-sum}
\begin{aligned}
 \sum_{k\ge1}\rho_{2^k}^\beta(u_k(t)-u_{k-1}(t))
 &=\sum_{n\ge0}\rho_{2^{n+4}}^\beta(\pi_{n+1}(A_{n+1})-\pi_n(A_n))\\
 &\le\frac{48}{\log8}\sum_{n\ge0}\tau(A_n)
 \le C(\tau(T)+F(T)),
\end{aligned}
\end{equation}
Also $u_0(T)+\sum_{k\ge1}(u_k(t)-u_{k-1}(t))=z(t)$.
By \eqref{eq:part-reserve}, \eqref{eq:part-root-F}, and
\eqref{eq:est-diam2}, this proves \eqref{eq:chop-main}.
For $t$ and $A$ in the original set, put
\[
 z'(t)=t_0+z(t-t_0),\qquad u_k'(A)=t_0+u_k(A-t_0).
\]
Then $t-z'(t)=(t-t_0)-z(t-t_0)$ and
$u'_k(t)-u'_{k-1}(t)=u_k(t-t_0)-u_{k-1}(t-t_0)$.
Translate the members of $\mathcal B_k$ by $t_0$ and retain the
notation $\mathcal B_k$ for the partitions of the original set $T$.
Write $z,u_k$ for $z',u'_k$.
Thus $u_0(T)=t_0$.

For \eqref{eq:chop-main-decomp} put
\[
 h_k(t)=u_k(t)-u_{k-1}(t),
 \qquad
 \delta_k(t)=\rho_{2^k}^\beta(h_k(t)).
\]
Put
\[
 J_k(t)=\{i:|h_{k,i}(t)|>2\delta_k(t)/2^k\},
 \quad h_k^{(1)}(t)=(h_k(t))_{J_k(t)},
 \quad h_k^{(2)}(t)=(h_k(t))_{J_k(t)^c}.
\]
By Lemma~\ref{lem:chop-threshold},
\begin{equation}\label{eq:part-increment-split}
 \|h_k^{(1)}(t)\|_{1,\beta}\le\delta_k(t),
 \qquad
 2^{k/2}\|h_k^{(2)}(t)\|_2
 +2^k\|h_k^{(2)}(t)\|_\infty
 \le4\delta_k(t).
\end{equation}

Put
\[
 a(t)=t-z(t)+\sum_{k\ge1}h_k^{(1)}(t),
 \qquad
 v(t)=\sum_{k\ge1}h_k^{(2)}(t).
\]
Both series converge absolutely in $\ell_2$ by
\eqref{eq:chop-main} and \eqref{eq:part-increment-split}.
The sum of the $h_k(t)$ is $z(t)-t_0$. Thus
\[
 t-t_0=a(t)+v(t).
\]
Choose $T'\subset T$ with $v:T'\to v(T)$ bijective. The partitions
of $v(T)$ are
\[
 \{v(B\cap T'):B\in\mathcal B_k,\ B\cap T'\ne\emptyset\}.
\]
For $s=v(t)$ with $t\in T'$, put
\[
 w_0(s)=0,\qquad w_k(s)=\sum_{j=1}^k h_j^{(2)}(t).
\]
The vector $w_k$ is constant on each member of the $k$th partition
of $v(T)$ and $w_k(s)\to s$ in $\ell_2$. By
\eqref{eq:chop-main} and \eqref{eq:part-increment-split},
\begin{equation}\label{eq:part-v-chain}
 \sup_{t\in T}\|a(t)\|_{1,\beta}
 +\sup_{t\in T}\sum_{k\ge1}
 \bigl(2^{k/2}\|h_k^{(2)}(t)\|_2+2^k\|h_k^{(2)}(t)\|_\infty\bigr)
 \le C S_\beta(T).
\end{equation}
Lemma~\ref{lem:lower-gamma-exp} applied to $(w_k)$ gives
\eqref{eq:chop-main-decomp}.

\section{The lower bound}\label{sec:lower}

We prove, for every nonempty $T\subset\ell_2$,
\begin{equation}\label{eq:lower-goal}
 \gamma^\Psi(T)\le C S_Y(T).
\end{equation}
We use a capped exponential process on a linear image $LT$ of $T$,
as in \cite{HWW}, and apply Theorem~\ref{thm:chop-main} to $LT$.

\subsection{\texorpdfstring{The functionals $\gamma^\theta$}{Partition functionals}}

For the variable $\varepsilon\min\{E,\beta\}$ with $1\le\beta\le\infty$,
the function in \eqref{eq:intro-Psi} is
\begin{equation}\label{eq:lower-psibeta}
 \psi_\beta(x)=
 \begin{cases}
 x^2/4,&0\le x\le2,\\
 2\beta-1,&x>2.
 \end{cases}
\end{equation}
When $\beta=\infty$, take $2\beta-1=+\infty$. In particular
$\psi_1(x)=\chi(x/2)$, where $\chi(x)=\min\{1,x^2\}$ for $x\ge0$.
For a vector $\beta=(\beta_i)$, write $\psi_\beta=(\psi_{\beta_i})_i$.

The integral in \eqref{eq:intro-Psi} over $(0,1)$ gives $\psi_1$.
For $1<u<b_i$ we separate the values of $\sigma_i(u)$ into
intervals $[s_j,2s_j)$, with
\begin{equation}\label{eq:lower-sj}
 s_j=2^{j+1},
 \qquad
 c_j=2^{-j}=\frac2{s_j},
 \qquad j\ge0,
\end{equation}
and
\begin{equation}\label{eq:lower-Aij}
 I_{ij}=\{u\in(1,b_i):s_j\le\sigma_i(u)<2s_j\},
 \qquad
 w_{ij}=s_j|I_{ij}|.
\end{equation}
Here $|I_{ij}|$ is the length of $I_{ij}$. Put
\begin{equation}\label{eq:lower-betaij}
 \beta_{ij}=\max\left\{1,\frac{w_{ij}+1}{2}\right\}.
\end{equation}
With $J=\{(i,j):i\ge1,\ j\ge0\}$ define $L:\ell_2\to\ell_2\oplus\ell_2(J)$ by
\begin{equation}\label{eq:lower-L}
 Lt=\bigl(t,(c_jt_i)_{(i,j)\in J}\bigr),
 \qquad
 \|Lt\|_2^2=\left(1+\sum_{j\ge0}c_j^2\right)\|t\|_2^2.
\end{equation}
Put $\widetilde\beta=((1)_{i\ge1},(\beta_{ij})_{(i,j)\in J})$ and take
independent variables $\xi_i^{(1)}$ and $\xi_{ij}^{(\beta_{ij})}$
with the laws in \eqref{eq:intro-cap}. For $t\in T$,
\[
 \left\langle Lt,\xi^{(\widetilde\beta)}\right\rangle
 =\sum_{i\ge1}t_i\xi_i^{(1)}
   +\sum_{(i,j)\in J}c_jt_i\xi_{ij}^{(\beta_{ij})},
 \qquad\text{in }L^2.
\]
We shall prove that
\[
 \Psi_i(x)\le2\left(\psi_1(x)+\sum_{j\ge0}\psi_{\beta_{ij}}(c_jx)\right),
 \qquad S_{\widetilde\beta}(LT)\le C S_Y(T).
\]

Let $\theta=(\theta_i)_{i\ge1}$, where each
$\theta_i:[0,\infty)\to[0,\infty]$ is nondecreasing and
$\theta_i(0)=0$. For $u>0$ put
\begin{equation}\label{eq:lower-Dtheta}
 D_u^\theta(s,t)=\sum_i\theta_i(|s_i-t_i|/u),
\end{equation}
and
\begin{equation}\label{eq:lower-initial-diameter}
 \Delta_\theta(T)=
 \sup_{s,t\in T}\int_0^\infty\min\{1,D_u^\theta(s,t)\}\,du.
\end{equation}
Define
\begin{equation}\label{eq:lower-gamma-definition}
 \gamma^\theta(T)=
 \inf\sup_{t\in T}\sum_{k\ge0}2^k r_k(\cA_k(t)),
\end{equation}
where the infimum is over admissible sequences $(\cA_k)$ of
partitions of $T$ and numbers $r_k(A)\ge0$, $A\in\cA_k$, with
\[
 \begin{gathered}
 r_0(T)\ge\Delta_\theta(T),\qquad
 r_{k+1}(B)\le r_k(A)\quad(B\in\cA_{k+1},\ B\subset A),\\
 D_{r_k(A)}^\theta(s,t)\le2^k\quad(s,t\in A,\ r_k(A)>0),
 \qquad |A|=1\quad\text{if }r_k(A)=0.
 \end{gathered}
\]
Write $r_k(t)=r_k(\cA_k(t))$ and
$D_u=D_u^\Psi$. For $j\in\Z$, $D_{r^{-j}}(s,t)=\phi_j(s,t)$.
We use the same definitions on any countable set of coordinates.

For $\theta=\Psi$ and $\theta=\psi_\beta$ we have
$\theta_i(x)\wedge1=(x^2/4)\wedge1$. Since
$(\sum_i a_i)\wedge1=(\sum_i(a_i\wedge1))\wedge1$ for $a_i\ge0$,
\[
 D_u^\theta(s,t)\wedge1
 =\left(\frac{\|s-t\|_2^2}{4u^2}\right)\wedge1.
\]
Integrating gives
\begin{equation}\label{eq:lower-initial-euclidean}
 \int_0^\infty \bigl(D_u^\theta(s,t)\wedge1\bigr)\,du
 =\int_0^\infty\left(\frac{\|s-t\|_2^2}{4u^2}\wedge1\right)\,du
 =\|s-t\|_2 .
\end{equation}
Thus $\Delta_\Psi(T)=\Delta_{\psi_\beta}(T)=\Delta(T,d_2)$.
Take partitions and numbers $r_k(A)$ satisfying the conditions in
\eqref{eq:lower-gamma-definition} for $\theta=\Psi$.
For $r_k(A)>0$ choose $j_k(A)\in\Z$ with
$r_k(A)\le r^{-j_k(A)}<rr_k(A)$, and take $j_k(A)=+\infty$ if $r_k(A)=0$.
Then $\phi_{j_k(A)}(s,t)\le2^k$ for $s,t\in A$, and
$j_{k+1}(B)\ge j_k(A)$ for $B\subset A$. The sum in
\eqref{eq:intro-gamma} is bounded by $r$ times that in
\eqref{eq:lower-gamma-definition}.
Conversely, the choices $r_k(A)=r^{-j_k(A)}$ in
\eqref{eq:intro-gamma} are allowed in
\eqref{eq:lower-gamma-definition}. Hence, for every nonempty $T$,
\begin{equation}\label{eq:lower-discrete-continuous}
 \gamma^\Psi(T)\le\gamma_Y(T)\le r\gamma^\Psi(T).
\end{equation}

\subsection{\texorpdfstring{A bound for $\gamma^{\psi_\beta}$}{The partition functional for capped exponentials}}

To use $T-t_0=(a+v)(T)$ from Theorem~\ref{thm:chop-main},
we first estimate $\gamma^\theta$ on the image of a sum of two maps.

\begin{lemma}\label{lem:lower-glue}
For maps $f,g:T\to\ell_2$,
\begin{equation}\label{eq:lower-glue-map}
 \gamma^\theta((f+g)(T))
 \le 8\bigl(\gamma^\theta(f(T))+\gamma^\theta(g(T))\bigr).
\end{equation}
\end{lemma}

\begin{proof}
Since $|x_i-z_i|\le2\max\{|x_i-y_i|,|y_i-z_i|\}$, we have
\begin{equation}\label{eq:lower-dyadic-triangle}
 D_{2u}^\theta(x,z)\le D_u^\theta(x,y)+D_u^\theta(y,z).
\end{equation}
In the same way,
\[
 D_{2u}^\theta(x+x',y+y')\le D_u^\theta(x,y)+D_u^\theta(x',y').
\]
The monotonicity in $u$ gives, for $s,t\in V$ and $u>0$,
\[
 u\min\{1,D_u^\theta(s,t)\}\le\Delta_\theta(V).
\]
Put $U=(f+g)(T)$ and choose $t(x)\in T$ with
$f(t(x))+g(t(x))=x$ for each $x\in U$.
Take $(\cA_k,r_k)$ and $(\mathcal B_k,s_k)$ as in
\eqref{eq:lower-gamma-definition} for $f(T)$ and $g(T)$.
Let $\mathcal C_{k+1}$ consist of the nonempty sets
\[
 C_{A,B}=\{x\in U:f(t(x))\in A,\ g(t(x))\in B\},
 \quad A\in\cA_k,\quad B\in\mathcal B_k.
\]
Set $r'_{k+1}(C_{A,B})=2\max\{r_k(A),s_k(B)\}$.
There are at most $2^{2^{k+1}}$ such sets. Put
\[
 u=\max\{r_k(A),s_k(B)\}.
\]
If $u>0$ and $x,y\in C_{A,B}$, then
\[
 D_{2u}^\theta(x,y)
 \le D_u^\theta(f(t(x)),f(t(y)))+D_u^\theta(g(t(x)),g(t(y)))
 \le2^{k+1}.
\]
If $u=0$, then $|A|=|B|=1$, and hence $|C_{A,B}|=1$.

Integrating the estimate for sums, we get
\[
 \Delta_\theta(U)\le2\bigl(\Delta_\theta(f(T))+\Delta_\theta(g(T))\bigr).
\]
Put $\mathcal C_0=\{U\}$ and $r'_0(U)=4(r_0(f(T))+s_0(g(T)))$.
Then $r'_0(U)\ge2\Delta_\theta(U)$ and, if $r'_0(U)>0$,
$D_{r'_0(U)}^\theta(x,y)\le1/2$ for $x,y\in U$.
If $r'_0(U)=0$, then $|f(T)|=|g(T)|=|U|=1$.
The definition gives $r'_{k+1}(\mathcal C_{k+1}(x))\le r'_k(\mathcal C_k(x))$
for $k\ge0$. We obtain
\[
 \sup_{x\in U}\sum_{k\ge0}2^kr'_k(\mathcal C_k(x))
 \le r'_0(U)+4\sup_{t\in T}\sum_{k\ge0}2^k
 \bigl(r_k(\cA_k(f(t)))+s_k(\mathcal B_k(g(t)))\bigr).
\]
Taking infima gives \eqref{eq:lower-glue-map}.
\end{proof}

For $a(T)$ we will use the weighted $\ell_1$ bound in the
following estimate.

\begin{lemma}\label{lem:lower-centre}
Let $T\subset\ell_2$ and $t_0\in T$.
\begin{equation}\label{eq:lower-centre}
 \gamma^\theta(T)\le
 C\sup_{t\in T}\int_0^\infty D_u^\theta(t,t_0)\,du.
\end{equation}
\end{lemma}

\begin{proof}
Put
\[
 \Lambda=\sup_{t\in T}\int_0^\infty D_u^\theta(t,t_0)\,du.
\]
There is nothing to prove if $\Lambda=\infty$. Suppose first that
$0<\Lambda<\infty$. Let $p_k=2^k$ and
\[
 \lambda_k(t)=\sup\{u>0:D_u^\theta(t,t_0)>p_k\},
\]
where $\sup\varnothing=0$. If $u>\lambda_k(t)$, then
$D_u^\theta(t,t_0)\le p_k$. By monotonicity,
\begin{equation}\label{eq:lower-centre-lambda}
 p_k\lambda_k(t)\le\Lambda,\qquad
 \lambda_{k+1}(t)\le\lambda_k(t),\qquad
 \sup_{t\in T}\sum_{k\ge0}p_k\lambda_k(t)\le2\Lambda.
\end{equation}
For the sum, we use
\[
 \begin{aligned}
 \sum_{k\ge0}p_k\lambda_k(t)
 &=\int_0^\infty\sum_{k:u<\lambda_k(t)}p_k\,du\\
 &\le\int_0^\infty\sum_{k:D_u^\theta(t,t_0)>p_k}p_k\,du
 \le2\int_0^\infty D_u^\theta(t,t_0)\,du.
 \end{aligned}
\]
Put $\varepsilon_k=\Lambda p_k^{-2}$. By \eqref{eq:lower-centre-lambda},
$\lambda_k(t)\le p_k\varepsilon_k$. For $0\le\ell\le p_k$ set
\[
 B_{k,\ell}=
 \{t\in T:\ell\varepsilon_k\le\lambda_k(t)
                   <(\ell+1)\varepsilon_k\}.
\]
Let $\mathcal P_k$ consist of the nonempty sets
$\bigcap_{j=0}^k B_{j,m_j}$ with $0\le m_j\le p_j$. Then
\begin{equation}\label{eq:lower-centre-card}
 |\mathcal P_k|\le\prod_{j=0}^k(p_j+1)
 \le2^{(k+1)(k+2)/2}\le2^{2^{k+2}}.
\end{equation}
Define
\[
 \begin{gathered}
 \cA_0=\cA_1=\{T\},\qquad
 \cA_{k+2}=\mathcal P_k,\\
 r_0(T)=r_1(T)=8\Lambda,\qquad
 r_{k+2}(\mathcal P_k(t))
 =2(\ell+2)\varepsilon_k\quad(t\in B_{k,\ell}).
 \end{gathered}
\]
For $t\in B_{k,\ell}$ this choice gives
\begin{equation}\label{eq:lower-centre-scale}
 2\bigl(\lambda_k(t)+\varepsilon_k\bigr)
 <r_{k+2}(\mathcal P_k(t))
 \le2\bigl(\lambda_k(t)+2\varepsilon_k\bigr).
\end{equation}
For $A\in\cA_{k+2}$ and $s,t\in A$, put $u=r_{k+2}(A)/2$.
By \eqref{eq:lower-centre-scale}, $u>\lambda_k(s)$ and $u>\lambda_k(t)$.
Hence \eqref{eq:lower-dyadic-triangle} gives
\[
 D_{r_{k+2}(A)}^\theta(s,t)\le2p_k\le2^{k+2}.
\]
Furthermore, $\varepsilon_{k+1}=\varepsilon_k/4$ and
$\lambda_{k+1}\le\lambda_k$ imply
\[
 r_{k+3}(\cA_{k+3}(t))
 \le2\lambda_{k+1}(t)+4\varepsilon_{k+1}
 <2\lambda_k(t)+2\varepsilon_k
 <r_{k+2}(\cA_{k+2}(t)).
\]
Using $t_0$ in \eqref{eq:lower-dyadic-triangle},
\[
 \Delta_\theta(T)
 \le2\sup_{s,t\in T}\int_0^\infty
  \min\{1,D_u^\theta(s,t_0)+D_u^\theta(t,t_0)\}\,du
 \le4\Lambda.
\]
For $s,t\in T$ the bound
$u\min\{1,D_u^\theta(s,t)\}\le\Delta_\theta(T)$ at $u=8\Lambda$
gives $D_{8\Lambda}^\theta(s,t)\le1/2$. Also
$r_2(\cA_2(t))\le6\Lambda\le r_1(T)$. Finally, \eqref{eq:lower-centre-lambda} and
\eqref{eq:lower-centre-scale} give
\[
 \sup_{t\in T}\sum_{k\ge0}2^{k+2}
 r_{k+2}(\cA_{k+2}(t))
 \le8\sup_{t\in T}\sum_{k\ge0}p_k
 (\lambda_k(t)+2\varepsilon_k)
 \le C\Lambda.
\]
Since $r_0(T)+2r_1(T)=24\Lambda$, this gives
\eqref{eq:lower-centre}. If $\Lambda=0$, apply the construction
with any positive upper bound for $\Lambda$.
\end{proof}

For $v(T)$ we will use \eqref{eq:part-v-chain} and the following lemma.

\begin{lemma}\label{lem:lower-gamma-exp}
The equivalence \eqref{eq:intro-exp-chain} holds. Moreover, let
$(\cA_k)_{k\ge0}$ be admissible partitions of a nonempty
set $S\subset\ell_2$. Let $u_k(A)\in\ell_2$ for $A\in\cA_k$.
Put $u_k(s)=u_k(\cA_k(s))$ and suppose $u_k(s)\to s$ in
$\ell_2$ for $s\in S$.
Then
\begin{equation}\label{eq:lower-gamma-exp}
 \gamma^{\psi_\infty}(S)+\gexp(S)
 \le C\sup_{s\in S}\sum_{k\ge1}
 \bigl(2^{k/2}\|u_k(s)-u_{k-1}(s)\|_2
       +2^k\|u_k(s)-u_{k-1}(s)\|_\infty\bigr).
\end{equation}
\end{lemma}

\begin{proof}
For finite $S$, choose admissible sequences $(\mathcal B_k)$ and $(\mathcal C_k)$
for $\gamma_2(S,d_2)$ and $\gamma_1(S,d_\infty)$, respectively,
with sums at most twice these quantities. For all sufficiently
large $k$, take $\mathcal B_k=\mathcal C_k=\{\{t\}:t\in S\}$.
Put
\[
 \cA_0=\{S\},\qquad
 \cA_{k+1}
 =\{B\cap C:B\in\mathcal B_k,\ C\in\mathcal C_k,
                 \ B\cap C\ne\emptyset\}.
\]
Then $|\cA_{k+1}|\le|\mathcal B_k|\,|\mathcal C_k|
\le2^{2^{k+1}}$, so $(\cA_k)$ is admissible.
Choose $u_k(A)\in A$ for each $A\in\cA_k$. Then
$u_k(\{t\})=t$, and for $q=2,\infty$,
\[
 \|u_k(t)-u_{k-1}(t)\|_q\le\Delta(\cA_{k-1}(t),d_q).
\]
For $k\ge2$ this diameter is bounded by
$\Delta(\mathcal B_{k-2}(t),d_2)$ when $q=2$ and
$\Delta(\mathcal C_{k-2}(t),d_\infty)$ when $q=\infty$.
For $k=1$ use $\Delta(S,d_q)$. Summing these inequalities shows
that the infimum in \eqref{eq:intro-exp-chain} is at most $C\gexp(S)$.

For \eqref{eq:lower-gamma-exp}, take $(\cA_k,u_k)$ as in the
statement and put
\[
 K=\sup_{t\in S}\sum_{k\ge1}
 \bigl(2^{k/2}\|u_k(t)-u_{k-1}(t)\|_2+
             2^k\|u_k(t)-u_{k-1}(t)\|_\infty\bigr).
\]
We may assume $K<\infty$. For $p_k=2^k$, put
\begin{equation}\label{eq:lower-bk}
 b_k(t)=\sqrt{p_k}\,\|t-u_k(t)\|_2
       +p_k\|t-u_k(t)\|_\infty.
\end{equation}
Since
\[
 t-u_k(t)=\sum_{l>k}\bigl(u_l(t)-u_{l-1}(t)\bigr),
\]
we can sum first over $k<l$, using
$\sum_{k<l}\sqrt{p_k}\le C\sqrt{p_l}$ and $\sum_{k<l}p_k\le p_l$.
This gives
\begin{equation}\label{eq:lower-bksum}
 \sup_{t\in S}\sum_{k\ge0}b_k(t)\le CK.
\end{equation}
Set
\[
 B=\sup_{t\in S}\sum_{k\ge0}b_k(t).
\]
If $B=0$, then $S=\{u_0(S)\}$ and both functionals vanish.
Assume $B>0$. We construct partitions $(\mathcal P_k)$ such that
each member $A$ of $\mathcal P_k$ is contained in a member of
$\cA_k$ and $|b_k(s)-b_k(t)|\le B/p_k$ for $s,t\in A$.
For $j\ge0$ put
\[
 m_j(t)=\left\lfloor\frac{p_jb_j(t)}{B}\right\rfloor.
\]
Two points $s,t$ belong to the same member of $\mathcal P_k$ if
$\cA_k(s)=\cA_k(t)$ and $m_j(s)=m_j(t)$ for $0\le j\le k$.
Since
$0\le m_j(t)\le p_j$,
\begin{equation}\label{eq:lower-refined-card}
 |\mathcal P_k|
 \le |\cA_k|\prod_{j\le k}(p_j+1)
 \le \exp(Cp_k).
\end{equation}
For $A\in\mathcal P_k$ put
\begin{equation}\label{eq:lower-rk}
 r_k(A)=
 \max\left\{
 \Delta(A,d_\infty),\,
 \frac{\Delta(A,d_2)}{\sqrt{p_k}}
 \right\}.
\end{equation}
Since the partitions are increasing and $p_{k+1}\ge p_k$,
\[
 r_{k+1}(B)\le r_k(A)
 \qquad
 (B\subset A,\;
 B\in\mathcal P_{k+1},\ A\in\mathcal P_k).
\]

If $t\in A\in\mathcal P_k$, then $u_k$ is constant on $A$ and
\[
 |b_k(s)-b_k(t)|\le\frac{B}{p_k},
 \qquad s\in A.
\]
For $s,u\in A$ the triangle inequalities with the common vector
$u_k(t)=u_k(s)=u_k(u)$ give
\[
 \sqrt{p_k}\|s-u\|_2+p_k\|s-u\|_\infty
 \le b_k(s)+b_k(u)\le2b_k(t)+\frac{2B}{p_k}.
\]
Consequently
\begin{equation}\label{eq:lower-rk-bound}
 p_k r_k(A)
 \le2b_k(t)+\frac{2B}{p_k}.
\end{equation}
Moreover, if $r_k(A)>0$, then for $s,t\in A$,
\[
 \|s-t\|_\infty\le r_k(A),
 \qquad
 \|s-t\|_2\le\sqrt{p_k}\,r_k(A),
\]
and hence
\begin{equation}\label{eq:lower-rk-dist}
 D_{r_k(A)}^{\psi_\infty}(s,t)\le p_k.
\end{equation}
If $r_k(A)=0$, then $\Delta(A,d_2)=0$, so $|A|=1$.
It follows from \eqref{eq:lower-bksum} and
\eqref{eq:lower-rk-bound} that
\begin{equation}\label{eq:lower-rk-sum}
 \sup_{t\in S}
 \sum_{k\ge0}p_k r_k(\mathcal P_k(t))
 \le CB
 \le CK.
\end{equation}

By \eqref{eq:lower-refined-card}, we may choose a fixed integer
$h\ge1$ so that $|\mathcal P_k|\le2^{2^{k+h}}$ for every $k$.
Put
\[
 \begin{aligned}
  \mathcal Q_j&=\{S\},&s_j(S)&=2B &&(0\le j<h),\\
  \mathcal Q_{k+h}&=\mathcal P_k,
     &s_{k+h}(A)&=r_k(A) &&(k\ge0).
 \end{aligned}
\]
Then $(\mathcal Q_j)$ is admissible.
For $j=k+h$ and $A\in\mathcal Q_j$ with $s_j(A)>0$,
\eqref{eq:lower-rk-dist} gives $D_{s_j(A)}^{\psi_\infty}(s,t)
\le2^k\le2^j$ on $A$. Since $\|t-u_0(S)\|_2\le b_0(t)\le B$,
\[
 \Delta_{\psi_\infty}(S)=\Delta(S,d_2)\le2B,
 \qquad
 D_{2B}^{\psi_\infty}(s,t)\le1\quad(s,t\in S).
\]
For $A\in\mathcal P_0$ we have $r_0(A)\le\Delta(S,d_2)\le2B$.
Thus $s_{j+1}(\mathcal Q_{j+1}(t))\le s_j(\mathcal Q_j(t))$ for all $j$.
By \eqref{eq:lower-rk-sum},
\[
 \begin{aligned}
 \gamma^{\psi_\infty}(S)
 &\le (2^h-1)2B
   +2^h\sup_{t\in S}\sum_{k\ge0}2^kr_k(\mathcal P_k(t))\\
 &\le CK.
 \end{aligned}
\]
To bound $\gexp(S)$, note that \eqref{eq:lower-rk} gives
\[
 2^{k/2}\Delta(A,d_2)+2^k\Delta(A,d_\infty)\le2p_k r_k(A),
 \qquad A\in\mathcal P_k.
\]
Using $(\mathcal Q_j)$ in both diameter sums gives
\[
 \gexp(S)\le C\big(\Delta(S,d_2)+\Delta(S,d_\infty)\big)
 +C\sup_{t\in S}\sum_{k\ge0}p_kr_k(\mathcal P_k(t))
 \le CB,
\]
where $C$ absorbs the fixed shift $h$.
This proves \eqref{eq:lower-gamma-exp}. Given a sequence in
\eqref{eq:intro-exp-chain}, take $\cA_k=\{\{s\}:s\in S\}$ and
$u_k(s)=s$ for $k>N$. Apply \eqref{eq:lower-gamma-exp} and take the
infimum over the sequences in \eqref{eq:intro-exp-chain}.
\end{proof}

We now apply these estimates to $a(T)$ and $v(T)$ in
Theorem~\ref{thm:chop-main}.

\begin{theorem}\label{thm:lower-capped}
For every nonempty $T\subset\ell_2$,
\begin{equation}\label{eq:lower-capped}
 \gamma^{\psi_\beta}(T)\le C S_\beta(T).
\end{equation}
\end{theorem}

\begin{proof}
Assume $S_\beta(T)<\infty$ and $|T|>1$. Take $a$ and $v$ from
Theorem~\ref{thm:chop-main}, and put
\[
 M_0=\sup_{t\in T}\|a(t)\|_{1,\beta}\le CS_\beta(T).
\]

Fix $a^0\in a(T)$. For $\beta_i=\infty$ we have
$a_i(t)=a_i^0=0$ for all $t$. By \eqref{eq:lower-psibeta},
\begin{equation}\label{eq:lower-psi-integral}
 \int_0^\infty\psi_\beta(u)u^{-2}\,du=\beta.
\end{equation}
Use Lemma~\ref{lem:lower-centre} and make the change of variable
$x=|a_i(t)-a_i^0|/u$ in each term with $a_i(t)\ne a_i^0$.
This gives
\begin{align}
 \gamma^{\psi_\beta}(a(T))
 &\le
 C\sup_{t\in T}
 \sum_i |a_i(t)-a_i^0|
 \int_0^\infty\psi_{\beta_i}(u)u^{-2}\,du \notag\\
 &\le
 C\sup_{t\in T}\sum_i
 \beta_i|a_i(t)-a_i^0|
 \le CM_0.
 \label{eq:lower-a}
\end{align}
For $v(T)$ use the partitions and vectors in the proof of
Theorem~\ref{thm:chop-main}. By Lemma~\ref{lem:lower-gamma-exp},
\eqref{eq:part-v-chain}, and $\psi_{\beta_i}\le\psi_\infty$,
\begin{equation}\label{eq:lower-v}
 \gamma^{\psi_\beta}(v(T))
 \le\gamma^{\psi_\infty}(v(T))
 \le CS_\beta(T).
\end{equation}
Therefore, by translation invariance and
Lemma~\ref{lem:lower-glue},
\[
 \gamma^{\psi_\beta}(T)
 \le
 C\left(
 \gamma^{\psi_\beta}(a(T))
 +\gamma^{\psi_\beta}(v(T))
 \right)
 \le CS_\beta(T).
\]
\end{proof}

\subsection{\texorpdfstring{Comparison of $\gamma^\Psi$ and $\gamma^{\psi_{\widetilde\beta}}$}{Comparison of partition functionals}}

\begin{lemma}\label{lem:lower-function-comparison}
For every nonempty $T\subset\ell_2$,
\begin{equation}\label{eq:lower-function-comparison}
 \gamma^\Psi(T)\le 3\gamma^{\psi_{\widetilde\beta}}(LT).
\end{equation}
\end{lemma}

\begin{proof}
By \eqref{eq:intro-Psi} and \eqref{eq:lower-Aij},
\[
 \Psi_i(x)=\psi_1(x)+\sum_{j\ge0}\int_{I_{ij}}
                       \sigma_i(u)\ind_{\{\sigma_i(u)<x\}}\,du.
\]
For $x\le s_j$ the integral over $I_{ij}$ vanishes. For $x>s_j$,
we have $c_j x>2$, so
\[
 \int_{I_{ij}}\sigma_i(u)\ind_{\{\sigma_i(u)<x\}}\,du
 \le2w_{ij}\le2\psi_{\beta_{ij}}(c_jx).
\]
It follows that
\[
 D_u^\Psi(s,t)\le2D_u^{\psi_{\widetilde\beta}}(Ls,Lt).
\]
Take $(\cA_k,r_k)$ as in \eqref{eq:lower-gamma-definition} for
$LT$ and $\theta=\psi_{\widetilde\beta}$.
Put
\[
 \begin{gathered}
 \mathcal B_0=\{T\},\qquad
 \mathcal B_{k+1}=\{L^{-1}(A):A\in\cA_k\},\\
 s_0(T)=r_0(LT),\qquad s_{k+1}(L^{-1}(A))=r_k(A).
 \end{gathered}
\]
By construction $s_{k+1}(B')\le s_k(B)$ for
$B'\in\mathcal B_{k+1}$, $B\in\mathcal B_k$, $B'\subset B$.
By \eqref{eq:lower-L} and \eqref{eq:lower-initial-euclidean},
$s_0(T)\ge\Delta(T,d_2)$, so $D_{s_0(T)}^\Psi\le1/4$ on $T\times T$
if $s_0(T)>0$. If $s_0(T)=0$, then $|T|=1$. For $B\in\mathcal B_{k+1}$ with $s_{k+1}(B)>0$ and $s,t\in B$,
we have $D_{s_{k+1}(B)}^\Psi(s,t)\le2^{k+1}$ by the comparison above.
If $s_{k+1}(B)=0$, then $|L(B)|=1$, and hence $|B|=1$ by injectivity of $L$. Thus these partitions and numbers satisfy
\eqref{eq:lower-gamma-definition}, and
\[
 \sum_{k\ge0}2^k s_k(\mathcal B_k(t))
 =r_0(LT)+2\sum_{k\ge0}2^k r_k(\cA_k(Lt))
 \le3\sum_{k\ge0}2^k r_k(\cA_k(Lt)).
\]
Take the supremum over $t$ and then the infimum in
\eqref{eq:lower-gamma-definition} for $LT$.
\end{proof}

\subsection{\texorpdfstring{Comparison of $S_Y(T)$ and $S_{\widetilde\beta}(LT)$}{Comparison of expected suprema}}

Put
\begin{equation}\label{eq:lower-Zi}
 Z_i=\xi_i^{(1)}+\sum_{j\ge0}c_j\xi_{ij}^{(\beta_{ij})}.
\end{equation}
For $p\ge1$,
\[
 \sum_{j\ge0}c_j\|\xi_{ij}^{(\beta_{ij})}\|_p
 \le Cp\sum_{j\ge0}c_j\le Cp.
\]
The series in \eqref{eq:lower-Zi} converges absolutely in $L^p$ for
every $p\ge1$. The $Z_i$ are independent and symmetric. By \eqref{eq:lower-L},
for $t\in T$ we have $\langle Lt,\xi^{(\widetilde\beta)}\rangle
=\sum_i t_iZ_i$ in $L^2$. Thus
\begin{equation}\label{eq:lower-expanded-process}
 S_{\widetilde\beta}(LT)=\E\sup_{t\in T}\sum_i t_iZ_i.
\end{equation}

\begin{lemma}\label{lem:lower-scalar-moment}
For $i\ge1$ and $p\ge1$,
\begin{equation}\label{eq:lower-scalar-moment}
 \|Z_i\|_p\le Cq_i(p).
\end{equation}
\end{lemma}

\begin{proof}
Fix $i$, and write $q=q_i$, $N=N_i$, $b=b_i$, $\sigma=\sigma_i$,
$I_j=I_{ij}$, $w_j=w_{ij}$, and $\beta_j=\beta_{ij}$.
For $\beta\ge1$,
\begin{equation}\label{eq:lower-capped-moment}
 \|\varepsilon\min(E,\beta)\|_p
 \le C\min\{p,\beta\}.
\end{equation}
Hence, by Minkowski's inequality, it is enough to prove
\begin{equation}\label{eq:lower-scalar-sum}
 1+
 \sum_{j\ge0}
 c_j\min\{p,\beta_j\}
 \le Cq(p).
\end{equation}

Since $\beta_j=1$ when $w_j<1$, we have
$\sum_{j:w_j<1}c_j\min\{p,\beta_j\}\le2$.
Put $u=q(p)$. For $w_j\ge1$,
\begin{equation}\label{eq:lower-bin-length}
 c_j\beta_j\le c_jw_j=2|I_j|.
\end{equation}
The intervals $I_j$ with $w_j\ge1$ and $I_j\subset(1,u]$
contribute at most $2(u-1)$. This treats all such terms if $u=b$.

If $u<b$, then $N(u)=p$. For an interval meeting $(u,b)$,
take $v\in I_j$ with $v>u$. By
convexity of $N$,
\[
 2s_j>\sigma(v)=2N'_+(v)
 \ge\frac{2N(u)}{u}
 =\frac{2p}{u}.
\]
Thus $s_j>p/u$, and the sum over these intervals is bounded by
\[
 \sum_{s_j>p/u}c_jp
 =
 \sum_{s_j>p/u}\frac{2p}{s_j}
 \le4u.
\]
Since $u\ge1$, this proves \eqref{eq:lower-scalar-sum}.
\end{proof}

For symmetric random variables $U$ and $V$, write
$U\preceq_{\rm cx}V$ if
\[
 \E f(U)\le\E f(V)
\]
for every convex function $f$ for which both expectations exist.

\begin{lemma}\label{lem:lower-convex-order}
Let $q:[0,\infty)\to[0,\infty)$ be nondecreasing and concave, with $q(1)=1$.
Let $K>0$. If $Z$ is symmetric and
\begin{equation}\label{eq:lower-convex-assumption}
 \|Z\|_p\le Kq(p),
 \qquad p\ge1,
\end{equation}
then
\begin{equation}\label{eq:lower-convex-order}
 Z\preceq_{\rm cx}CK\,\varepsilon q(E),
\end{equation}
where $E$ is a mean-one exponential variable and $\varepsilon$ is an
independent symmetric sign.
\end{lemma}

\begin{proof}
By Markov's inequality,
\[
 \P\bigl(|Z|>eKq(p)\bigr)\le e^{-p},
 \qquad p\ge1.
\]
Put $b=\sup_{p\ge1}q(p)$. If $b<\infty$, then $|Z|\le Kb$ almost
surely by \eqref{eq:lower-convex-assumption}. For
$eKq(1)\le x<eKb$ choose $p\ge1$ with $eKq(p)=x$. Since $q$ is
strictly increasing wherever $q(p)<b$,
\[
 \P(|Z|>x)\le e^{-p}\le e^{1-p}=\P(eKq(E+1)>x).
\]
For the other values of $x$,
\[
 \P(eKq(E+1)>x)=1\quad(x<eKq(1)),\qquad
 \P(|Z|>x)=0\quad(x\ge eKb).
\]
Thus, in stochastic order,
\begin{equation}\label{eq:lower-stochastic}
 |Z|\preceq_{\rm st}eKq(E+1).
\end{equation}
To pass from $q(E+1)$ to $q(E)$, use concavity in the form
\[
 q(v+1)-q(v)\le q(1)-q(0)\le1\qquad(v\ge0).
\]

Let $f$ be convex and put
\[
 g(x)=\frac{f(x)+f(-x)}2.
\]
Then $g$ is even, convex, and nondecreasing on $[0,\infty)$. From
\eqref{eq:lower-stochastic} and convexity,
\[
 \E f(Z)=\E g(|Z|)
 \le \E g(eKq(E)+eK)
 \le \tfrac12\E g(2eKq(E))+\tfrac12 g(2eK).
\]
Moreover,
\[
 \E q(E)\ge\P(E\ge1)=e^{-1},
\]
so Jensen's inequality gives
\[
 g(2eK)
 \le
 \E g\bigl(2e^2Kq(E)\bigr).
\]
Consequently,
\[
 \E f(Z)\le\E g\bigl(2e^2Kq(E)\bigr)
 =
 \E f\bigl(2e^2K\varepsilon q(E)\bigr).
\]
\end{proof}

By Lemmas~\ref{lem:lower-scalar-moment} and
\ref{lem:lower-convex-order}, $Z_i\preceq_{\rm cx}CY_i$.
Let $F\subset T$ be finite and nonempty. For each $m$, integrate
$\max_{t\in F}\sum_{i\le m}t_i x_i$ with respect to one coordinate
at a time. Independence and the convex order at each coordinate give
\[
 \E\max_{t\in F}\sum_{i\le m}t_i Z_i
 \le C\E\max_{t\in F}\sum_{i\le m}t_i Y_i.
\]
For $W=Y,Z$ write $W_t=\sum_i t_iW_i$. The second moments of
the $W_i$ are uniformly bounded, so
\[
 \left\|\max_{t\in F}W_t-\max_{t\in F}W_{t_{\{1,\ldots,m\}}}\right\|_1
 \le C\sum_{t\in F}\|t_{\{i>m\}}\|_2\xrightarrow[m\to\infty]{}0.
\]
Letting $m\to\infty$ and taking the supremum over $F$ gives
\begin{equation}\label{eq:lower-process-comparison}
 \E\sup_{t\in T}\sum_i t_iZ_i\le CS_Y(T).
\end{equation}

\subsection{Proof of the lower bound}

Apply Theorem~\ref{thm:lower-capped} to $LT$ with caps
$\widetilde\beta$. By Lemma~\ref{lem:lower-function-comparison},
\eqref{eq:lower-expanded-process} and \eqref{eq:lower-process-comparison},
\begin{equation}\label{eq:lower-final}
 \gamma^\Psi(T)
 \le3\gamma^{\psi_{\widetilde\beta}}(LT)
 \le C S_{\widetilde\beta}(LT)
 \le CS_Y(T).
\end{equation}

\section{The decomposition and the upper bound}\label{sec:upper}

Let $T\subset\ell_2$ be nonempty. We first prove
\begin{equation}\label{eq:upper-given-decomposition}
 S_Y(T)\le \E\sup_{a\in T_1}\sum_i|a_i|\,|Y_i|+C\gexp(T_2)
\end{equation}
for $t_0\in T$ and $T_1,T_2\subset\ell_2$ with $T-t_0\subset T_1+T_2$.
We then construct such a decomposition with
\[
 \E\sup_{a\in T_1}\sum_i|a_i|\,|Y_i|+\gexp(T_2)
 \le C\gamma^\Psi(T).
\]
Together with \eqref{eq:lower-final}, this also gives the reverse
inequality in \eqref{eq:intro-decomp}.

\subsection{The upper bound for a given decomposition}

For each $t\in T$ choose $a(t)\in T_1$ and $v(t)\in T_2$ with
$t-t_0=a(t)+v(t)$. On every finite subset of $T$,
\[
 Y_t-Y_{t_0}=Y_{a(t)}+Y_{v(t)},\qquad
 |Y_{a(t)}|\le\sum_i|a_i(t)|\,|Y_i|.
\]
Since $\E Y_{t_0}=0$, we obtain
\[
 S_Y(T)\le \E\sup_{a\in T_1}\sum_i|a_i|\,|Y_i|+S_Y(T_2).
\]
It remains to bound $S_Y(T_2)$.

\begin{lemma}\label{lem:upper-moment}
For $p\ge1$ and $h\in\ell_2$,
\begin{equation}\label{eq:upper-moment}
 \|\langle h,Y\rangle\|_p
 \le
 C\bigl(\sqrt p\,\|h\|_2+p\|h\|_\infty\bigr).
\end{equation}
\end{lemma}

\begin{proof}
For finitely supported $h$, use $q_i(u)\le1+u$ and apply the
contraction principle to the
signs for each fixed $(E_i)$ and integrate in $(E_i)$. We obtain
\[
 \|Y_h\|_p
 \le\left\|\sum_i h_i\varepsilon_i(1+E_i)\right\|_p
 \le\left\|\sum_i h_i\varepsilon_i\right\|_p
   +\left\|\sum_i h_i\varepsilon_iE_i\right\|_p.
\]
By the Khintchine inequality, $\|\sum_i h_i\varepsilon_i\|_p
\le C\sqrt p\,\|h\|_2$. For $p\ge2$ use \eqref{eq:est-bernstein}
with $\beta_i=\infty$ for all $i$ to bound
$\|\sum_i h_i\varepsilon_iE_i\|_p$.
For $1\le p<2$ use the bound at $p=2$. For $h\in\ell_2$ put
$h^{(m)}=h_{\{1,\ldots,m\}}$. By \eqref{eq:intro-process},
$Y_{h^{(m)}}\to Y_h$ in $L^2$. By Fatou's lemma along a subsequence,
\[
 \|Y_h\|_p\le\liminf_{m\to\infty}\|Y_{h^{(m)}}\|_p
 \le C\bigl(\sqrt p\,\|h\|_2+p\|h\|_\infty\bigr).
\]
\end{proof}

\begin{lemma}\label{lem:upper-exp}
For every nonempty $S\subset\ell_2$,
\begin{equation}\label{eq:upper-exp}
 S_Y(S)\le C\gexp(S).
\end{equation}
\end{lemma}

\begin{proof}
Let $F\subset S$ be finite and nonempty. Take
$(\cA_k,u_k)_{0\le k\le N}$ as in \eqref{eq:intro-exp-chain} for $F$.
For $A\in\cA_k$, $1\le k\le N$, let $A'\in\cA_{k-1}$
contain $A$ and put
\[
 w_k(A)=u_k(A)-u_{k-1}(A'),\qquad
 b_k(A)=2^{k/2}\|w_k(A)\|_2+2^k\|w_k(A)\|_\infty .
\]
By Lemma~\ref{lem:upper-moment}, $\|Y_{w_k(A)}\|_{2^k}\le Cb_k(A)$.
Put
\[
 M=\sup_{\substack{1\le k\le N,\ A\in\cA_k\\ b_k(A)>0}}
       \frac{|Y_{w_k(A)}|}{b_k(A)},
\]
with $M=0$ if every $b_k(A)=0$. For $u\ge4C$,
\[
 \P(M>u)\le\sum_{k\ge1}2^{2^k}(C/u)^{2^k}
 \le C'/u^2.
\]
Thus $\E M\le C$. Since $\E Y_{u_0(F)}=0$,
\[
 S_Y(F)=\E\max_{s\in F}\sum_{k=1}^N Y_{w_k(\cA_k(s))}
 \le \E M\sup_{s\in F}\sum_{k=1}^N b_k(\cA_k(s)).
\]
Taking the infimum gives $S_Y(F)\le C\gexp(F)\le C\gexp(S)$.
Take the supremum over $F$.
\end{proof}

Apply \eqref{eq:upper-exp} with $S=T_2$ to obtain
\eqref{eq:upper-given-decomposition}. For $Y_i=\xi_i^{(\beta_i)}$
we also have
\[
 S_\beta(T)\le\sup_{a\in T_1}\|a\|_{1,\beta}+C\gexp(T_2).
\]
Take the infimum over the decompositions in Theorem~\ref{thm:intro-decomp}.

\subsection{Partitions and the decomposition}

For $T=\{t_0\}$ the decomposition is given by $a(t_0)=v(t_0)=0$.
In the rest of the construction assume $|T|\ge2$ and
$\gamma^\Psi(T)<\infty$. By \eqref{eq:lower-gamma-definition}, choose admissible
partitions $(\cA_n)$ and numbers $r_n(A)$ such that
\begin{align}
 \cA_0&=\{T\},
 &|\cA_n|&\le2^{2^n},
 \label{eq:upper-part-card}\\
 r_{n+1}(B)&\le r_n(A)\quad(B\in\cA_{n+1},\ A\in\cA_n,\ B\subset A),
 &D_{r_n(A)}(s,t)&\le2^n\quad(s,t\in A,\ r_n(A)>0),
 \label{eq:upper-part-dist}
\end{align}
with $r_0(T)\ge\Delta(T,d_2)$ and $r_n(A)>0$ if $|A|>1$, and
\begin{equation}\label{eq:upper-part-sum}
 \sup_{t\in T}\sum_{n\ge0}2^nr_n(\cA_n(t))
 \le2\gamma^\Psi(T).
\end{equation}
Put
\begin{equation}\label{eq:upper-B}
 B=\sup_{t\in T}\sum_{n\ge0}
 2^nr_n(\cA_n(t)).
\end{equation}
For $A\in\cA_n$ choose $\pi_n(A)\in A$ and put $t_0=\pi_0(T)$. Write
\[
 \pi_n(t)=\pi_n(\cA_n(t)),
 \qquad
 r_n(t)=r_n(\cA_n(t)).
\]

For $A\in\cA_n$, $n\ge1$, let $A'\in\cA_{n-1}$ contain
$A$, and put
\begin{equation}\label{eq:upper-increment-data}
 r_A=r_{n-1}(A'),
 \qquad
 z_A=\pi_n(A)-\pi_{n-1}(A'),
 \qquad
 I_A=\{i:|z_{A,i}|>4r_A\}.
\end{equation}
The notation $r_A,z_A,I_A,h_A$ will refer to the pair $(n,A)$.

Using the interval projection \eqref{eq:chop-PQ}, put
\begin{equation}\label{eq:upper-hA}
 h_{A,i}=P_{[-4r_A,4r_A]}(z_{A,i})
 =\operatorname{sgn}(z_{A,i})\min\{4r_A,|z_{A,i}|\}.
\end{equation}
For $A\in\cA_n$ with $r_A>0$,
\begin{equation}\label{eq:upper-hA-size}
 \|h_A\|_\infty\le4r_A,
 \qquad
 \|h_A\|_2^2
 =16r_A^2\sum_i\chi\left(\frac{|z_{A,i}|}{4r_A}\right)
 \le16r_A^2\,2^{n-1}.
\end{equation}
Indeed, $\chi(|z_{A,i}|/(4r_A))\le\Psi_i(|z_{A,i}|/r_A)$ and
$\pi_n(A),\pi_{n-1}(A')\in A'$.
If $r_A=0$, then $h_A=0$, and both norm bounds in
\eqref{eq:upper-hA-size} hold. It follows that
\[
 \sum_{n\ge1}\|h_{\cA_n(t)}\|_2
 \le4\sum_{n\ge1}2^{(n-1)/2}r_{n-1}(t)\le4B.
\]
Put
\begin{equation}\label{eq:upper-v}
 v(t)=\sum_{n\ge1}h_{\cA_n(t)},
\end{equation}
where the series converges absolutely in $\ell_2$.
Set $a(t)=t-t_0-v(t)$. Thus
\begin{equation}\label{eq:upper-decomp-id}
 t-t_0=a(t)+v(t).
\end{equation}

We shall prove that
\begin{equation}\label{eq:upper-two-bounds}
 \E\sup_{t\in T}\sum_i|a_i(t)|\,|Y_i|\le CB,
 \qquad
 \gexp(v(T))\le CB.
\end{equation}
\subsection{\texorpdfstring{The vectors $a(t)$}{The vectors a(t)}}

We use the following lemma to estimate $a_i(t)$.

\begin{lemma}\label{lem:upper-scalar}
Let $\kappa>0$, $r_0\ge r_1\ge\cdots\ge0$ and $\sum_{n\ge0}r_n<\infty$.
Let $(y_n)_{n\ge0}$ be a sequence of real numbers and suppose
$|x-y_0|\le\kappa r_0$. Then
\begin{equation}\label{eq:upper-scalar}
 \left|
 x-y_0-\sum_{k\ge0}
 P_{[-2\kappa r_k,2\kappa r_k]}(y_{k+1}-y_k)
 \right|
 \le
 6\kappa\sum_{n\ge1}r_{n-1}
 \ind_{\{|x-y_n|>\kappa r_n\}}.
\end{equation}
\end{lemma}

\begin{proof}
Let
\[
 I=\{n\ge1:|x-y_n|>\kappa r_n\}.
\]
For $k,k+1\notin I$ we have
\[
 |y_{k+1}-y_k|\le2\kappa r_k,\qquad
 P_{[-2\kappa r_k,2\kappa r_k]}(y_{k+1}-y_k)=y_{k+1}-y_k.
\]

Let $\{p,\ldots,q\}$ be a maximal interval of integers contained in
$I$, with $q<\infty$.
Since $p-1,q+1\notin I$, we have
$|y_{q+1}-y_{p-1}|\le\kappa(r_{q+1}+r_{p-1})$.
Hence
\[
 \left|y_{q+1}-y_{p-1}-\sum_{k=p-1}^{q}
 P_{[-2\kappa r_k,2\kappa r_k]}(y_{k+1}-y_k)\right|
 \le\kappa(r_{q+1}+r_{p-1})+2\kappa\sum_{k=p-1}^q r_k
 \le6\kappa\sum_{k=p}^q r_{k-1}.
\]
For an integer $g\ge0$ with $g\notin I$, summing over the maximal
intervals contained in $\{1,\ldots,g-1\}\cap I$ gives
\[
 \left|y_g-y_0-\sum_{k<g}
    P_{[-2\kappa r_k,2\kappa r_k]}(y_{k+1}-y_k)\right|
 \le6\kappa\sum_{\substack{n<g\\n\in I}}r_{n-1}.
\]
Also,
\[
 \left|x-y_g-\sum_{k\ge g}
 P_{[-2\kappa r_k,2\kappa r_k]}(y_{k+1}-y_k)\right|
 \le |x-y_g|+2\kappa\sum_{k\ge g}r_k\le3\kappa\sum_{k\ge g}r_k .
\]
If $\{0,1,\ldots\}\setminus I$ is finite, take its largest member
$g$. Then $\sum_{k\ge g}r_k=\sum_{n>g,\,n\in I}r_{n-1}$.
Otherwise let $g\to\infty$ through the indices not in $I$.
This proves \eqref{eq:upper-scalar}.
\end{proof}

For $\kappa\ge2$ and $i\ge1$ put
\begin{equation}\label{eq:upper-akappa}
 a_i^{(\kappa)}(t)
 =
 t_i-t_{0,i}
 -
 \sum_{n\ge1}
 P_{[-2\kappa r_{n-1}(t),2\kappa r_{n-1}(t)]}
 \bigl(\pi_{n,i}(t)-\pi_{n-1,i}(t)\bigr).
\end{equation}
The series converges absolutely since the $n$th summand is bounded
in absolute value by $2\kappa r_{n-1}(t)$.

\begin{lemma}\label{lem:upper-a-integral}
For every $t\in T$,
\begin{equation}\label{eq:upper-a-integral}
 \sum_i|a_i^{(2)}(t)|
 +
 \sum_i\int_1^{b_i}
 |a_i^{(\sigma_i(u))}(t)|\,du
 \le CB.
\end{equation}
\end{lemma}

\begin{proof}
In Lemma~\ref{lem:upper-scalar} take
$x=t_i$, $y_n=\pi_{n,i}(t)$ and $r_n=r_n(t)$. The initial
condition follows from $|t_i-t_{0,i}|\le\Delta(T,d_2)\le r_0(T)$.
If $r_n(t)>0$, then
\begin{equation}\label{eq:upper-dist-to-rep}
 D_{r_n(t)}(t,\pi_n(t))\le2^n.
\end{equation}
If $r_n(t)=0$, then $\cA_n(t)=\{t\}$ and $\pi_n(t)=t$.
The indicators in the sums below vanish whenever $r_n(t)=0$.
Using $\kappa=2$ and \eqref{eq:intro-Psi-basic},
\[
 \sum_i|a_i^{(2)}(t)|
 \le
 C\sum_{n\ge1}r_{n-1}(t)
 \sum_i
 \ind_{\{|t_i-\pi_{n,i}(t)|>2r_n(t)\}}
 \le C\sum_{n\ge1}2^n r_{n-1}(t)
 \le CB.
\]

For $1<u<b_i$, use Lemma~\ref{lem:upper-scalar} with
$\kappa=\sigma_i(u)\ge2$. Tonelli's theorem gives
\begin{align*}
 &\sum_i\int_1^{b_i}|a_i^{(\sigma_i(u))}(t)|\,du\\
 &\quad\le
 C\sum_{n\ge1}r_{n-1}(t)
 \sum_i\int_1^{b_i}
 \sigma_i(u)
 \ind_{\{\sigma_i(u)r_n(t)
 <|t_i-\pi_{n,i}(t)|\}}\,du\\
 &\quad\le
 C\sum_{n\ge1}2^nr_{n-1}(t)
 \le CB,
\end{align*}
by \eqref{eq:intro-Psi} and
\eqref{eq:upper-dist-to-rep}.
\end{proof}

Put $U_i=q_i(E_i)$. By \eqref{eq:upper-hA}, \eqref{eq:upper-v},
and \eqref{eq:upper-akappa},
\begin{equation}\label{eq:upper-a}
 a(t)=a^{(2)}(t),
 \qquad (|Y_i|)_{i\ge1}\eqdist(U_i)_{i\ge1}.
\end{equation}
In the Bernoulli case $b_i=1$ and $U_i=1$, so
\eqref{eq:upper-a-integral} reduces to an $\ell_1$ bound.

\begin{proposition}\label{prop:upper-a}
For $a$ defined in \eqref{eq:upper-a},
\begin{equation}\label{eq:upper-a-positive}
 \E\sup_{t\in T}\sum_i|a_i(t)|U_i\le CB.
\end{equation}
\end{proposition}

\begin{proof}
Let $A\in\cA_n$, $n\ge1$, and let $A'$ be the member of
$\cA_{n-1}$ containing $A$. By \eqref{eq:intro-Psi-basic},
if $r_A>0$, then
\begin{equation}\label{eq:upper-IA}
 |I_A|
 \le D_{r_A}(\pi_n(A),\pi_{n-1}(A'))
 \le2^{n-1}.
\end{equation}
If $r_A=0$, then $|A'|=1$, so $z_A=0$ and $I_A=\emptyset$.

For $1<u<b_i$ put
\begin{equation}\label{eq:upper-HA}
 H_{A,i}(u)
 =
 \min\{|z_{A,i}|,2\sigma_i(u)r_A\}
 -
 \min\{|z_{A,i}|,4r_A\}.
\end{equation}
Then $H_{A,i}(u)=0$ for $i\notin I_A$ and
\begin{equation}\label{eq:upper-HA-bound}
 0\le H_{A,i}(u)\le2r_A\sigma_i(u).
\end{equation}
From \eqref{eq:upper-akappa},
\begin{equation}\label{eq:upper-a-compare}
 |a_i(t)|
 \le
 |a_i^{(\sigma_i(u))}(t)|
 +
 \sum_{n\ge1}
 H_{\cA_n(t),i}(u).
\end{equation}
Let $(x_i)_{i\ge1}$ be a sequence of finite numbers with $0\le x_i\le b_i$.
Then $|a_i(t)|x_i\le |a_i(t)|+\int_1^{x_i\vee1}|a_i(t)|\,du$.
By \eqref{eq:upper-a-compare} and \eqref{eq:upper-a-integral},
\begin{equation}\label{eq:upper-a-random-bound}
 \sum_i|a_i(t)|x_i
 \le
 CB+
 \sum_{n\ge1}\sum_i
 \int_1^{x_i\vee1}
 H_{\cA_n(t),i}(u)\,du.
\end{equation}
Take $x_i=U_i$. By \eqref{eq:intro-tail} and left continuity,
$N_i(U_i)\le E_i$. Hence
\[
 \int_1^{U_i\vee1}\sigma_i(u)\,du
 =V_i(U_i\vee1)-V_i(1)\le2E_i.
\]
By \eqref{eq:upper-HA-bound},
\begin{equation}\label{eq:upper-increment-positive}
 \sum_i\int_1^{U_i\vee1}H_{A,i}(u)\,du
 \le2r_A\sum_{i\in I_A}\int_1^{U_i\vee1}\sigma_i(u)\,du
 \le4r_A\sum_{i\in I_A}E_i.
\end{equation}

Put
\begin{equation}\label{eq:upper-Zstar}
 Z_*=
 \sup_{\substack{A\in\cA_n,\ n\ge1\\r_A>0}}
 \frac1{2^n}\sum_{i\in I_A}E_i.
\end{equation}
If all $r_A=0$, put $Z_*=0$.
By \eqref{eq:upper-IA},
\[
 \E\exp\left(\frac12\sum_{i\in I_A}E_i\right)
 =2^{|I_A|}
 \le2^{2^{n-1}}.
\]
Applying Markov's inequality at each $A$ gives
\[
 \P(Z_*>u)
 \le\sum_{n\ge1}|\cA_n|\,2^{2^{n-1}}e^{-u2^n/2}
 \le\sum_{n\ge1}e^{-cu2^n}\qquad(u\ge C),
\]
and hence $\E Z_*\le C$. Therefore, by
\eqref{eq:upper-a-random-bound} and
\eqref{eq:upper-increment-positive},
\[
 \E\sup_{t\in T}\sum_i|a_i(t)|U_i
 \le
 CB+
 C\E Z_*
 \sup_{t\in T}\sum_{n\ge1}2^nr_{\cA_n(t)}
 \le CB.
\]
\end{proof}

\subsection{\texorpdfstring{The vectors $v(t)$}{The vectors v(t)}}

Choose $T'\subset T$ with $v:T'\to v(T)$ bijective. Put
\[
 \begin{gathered}
 \mathcal B_k=\{v(A\cap T'):A\in\cA_k,\ A\cap T'\ne\emptyset\},\\
 u_0(v(t))=0,\qquad u_k(v(t))=\sum_{j=1}^k h_{\cA_j(t)}
 \quad(t\in T',\ k\ge1).
 \end{gathered}
\]
$u_k$ is constant on each member of $\mathcal B_k$ and $u_k(s)\to s$
in $\ell_2$. By
Lemma~\ref{lem:lower-gamma-exp} and \eqref{eq:upper-hA-size},
\begin{equation}\label{eq:upper-gamma-exp-v}
 \gexp(v(T))\le CB.
\end{equation}

\subsection{Proof of the decomposition and the upper bound}

\begin{proof}[Proof of Theorem~\ref{thm:intro-decomp}]
The bound for $S_Y(T)$ follows from \eqref{eq:upper-given-decomposition}.
To prove the reverse bound, it suffices to consider $S_Y(T)<\infty$.
If $|T|=1$, take $T_1=T_2=\{0\}$. Otherwise, by
\eqref{eq:upper-part-sum} and \eqref{eq:lower-final}, the construction
of $a(t)$ and $v(t)$ gives \eqref{eq:upper-two-bounds} with
\[
 B\le C\gamma^\Psi(T)\le C S_Y(T).
\]
Take $T_1=a(T)$ and $T_2=v(T)$. For \eqref{eq:intro-cap-rem}, use
\eqref{eq:chop-main-decomp} when $S_\beta(T)<\infty$; the other
bound follows from the upper bound for a given decomposition.
\end{proof}

\begin{theorem}\label{thm:upper}
For every nonempty $T\subset\ell_2$,
\begin{equation}\label{eq:upper-final}
 S_Y(T)\le C\gamma^\Psi(T).
\end{equation}
\end{theorem}

\begin{proof}
Assume $|T|>1$ and $\gamma^\Psi(T)<\infty$. Take
$T_1=a(T)$ and $T_2=v(T)$ in
\eqref{eq:upper-given-decomposition}. By \eqref{eq:upper-two-bounds}
and \eqref{eq:upper-part-sum},
\[
 S_Y(T)\le CB\le C\gamma^\Psi(T).
\]
\end{proof}

\begin{proof}[Proof of Theorem~\ref{thm:intro-main}]
By \eqref{eq:lower-discrete-continuous} and \eqref{eq:lower-final},
$\gamma_Y(T)\le r\gamma^\Psi(T)\le CS_Y(T)$. By
\eqref{eq:upper-final} and \eqref{eq:lower-discrete-continuous},
$S_Y(T)\le C\gamma^\Psi(T)\le C\gamma_Y(T)$.
\end{proof}

\section{Convex hulls}\label{sec:hulls}

We use the partitions and vectors constructed in Section~\ref{sec:upper}.
For a subset $K$ of a real Hilbert space, its polar is
$K^\circ=\{x:\langle x,y\rangle\le1\text{ for all }y\in K\}$.

\begin{corollary}\label{cor:countable-main}
Let $T\subset\ell_2$ be nonempty and $S_Y(T)<\infty$.
For each $\theta\ge1$ there are $s_j\in\ell_2$, $j\ge1$,
such that
\begin{equation}\label{eq:countable-hull}
 T-T\subset\clconv_{\ell_2}\{\pm s_j:j\ge1\},
 \qquad \|Y_{s_j}\|_{\theta\ell_j}\le C_\theta S_Y(T).
\end{equation}
\end{corollary}

\begin{proof}
If $|T|=1$, take $s_j=0$ for all $j$. Otherwise use the
partitions and vectors of Section~\ref{sec:upper}, with
$B\le C S_Y(T)$. For $A\in\cA_n$, $n\ge1$, we have
$|I_A|\le2^{n-1}$ by \eqref{eq:upper-IA}. Identify
$\R^{I_A}$ with $\{x\in\ell_2:x=x_{I_A}\}$ and put
\[
 K_A=\left\{x\in\R^{I_A}:|x_i|\le q_i\bigl((2^n+1)/2\bigr)
 \ \text{for }i\in I_A,\quad \sum_{i\in I_A}V_i(x_i)\le2^n\right\}.
\]
Left continuity gives $N_i(q_i(p))\le p$ for $p\ge0$. For $p\ge1$,
$V_i$ is even, continuous and convex on $[-q_i(p),q_i(p)]$. On its
domain, $V_i(x_i)\le 2^n$ implies $|x_i|\le q_i((2^n+1)/2)$.
Also $q_i(p)\ge1$ and $V_i(x_i)=x_i^2$ for $|x_i|\le1$.
It follows that $K_A$ is an unconditional convex body containing
the Euclidean unit ball of $\R^{I_A}$.

For $I_A\ne\emptyset$ take the polar $K_A^\circ$ in $\R^{I_A}$.
Choose $W_A\subset K_A^\circ$ maximal with
\[
 w-w'\notin\tfrac12 K_A^\circ\qquad(w,w'\in W_A,\ w\ne w').
\]
For distinct $w,w'\in W_A$,
\[
 (w+\tfrac14 K_A^\circ)\cap(w'+\tfrac14 K_A^\circ)=\emptyset,\qquad
 \bigcup_{w\in W_A}(w+\tfrac14 K_A^\circ)\subset\tfrac54 K_A^\circ.
\]
By volume and maximality,
\[
 |W_A|\le5^{|I_A|},\qquad K_A^\circ\subset W_A+\tfrac12 K_A^\circ.
\]
By the second inclusion,
\begin{equation}\label{eq:hull-net}
 \sup_{y\in K_A^\circ}|\langle y,x\rangle|
 \le2\max_{w\in W_A}|\langle w,x\rangle|,\qquad x\in\R^{I_A}.
\end{equation}
Take $W_A=\emptyset$ if $I_A=\emptyset$.

By \eqref{eq:intro-quantile} take $Y_i=\varepsilon_i q_i(E_i)$.
For $I_A\ne\emptyset$ put
\[
 M_A=\max\left\{1,\,2^{-n}\sum_{i\in I_A}V_i(Y_i)\right\}.
\]
Since $V_i(Y_i/M_A)\le V_i(Y_i)/M_A$, we have
$(Y_i)_{i\in I_A}/M_A\in K_A$. Thus for $w\in K_A^\circ$,
\[
 |Y_w|\le M_A\le1+2^{-n}\sum_{i\in I_A}(1+2E_i),
\]
since $N_i(q_i(E_i))\le E_i$. By
$\E\exp(\frac12\sum_{i\in I_A}E_i)=2^{|I_A|}$, for $u\ge0$,
\[
 \P\left(\sum_{i\in I_A}E_i>2|I_A|\log 2+2u\right)\le e^{-u}.
\]
Hence
\[
 \left\|\sum_{i\in I_A}E_i\right\|_p\le C(|I_A|+p),\qquad p\ge1.
\]
It follows from \eqref{eq:upper-IA} that
\begin{equation}\label{eq:hull-w-moment}
 \|Y_w\|_{p}\le C\bigl(1+2^{-n}p\bigr),\qquad p\ge1,\quad w\in W_A.
\end{equation}

Put
\[
 K=\clconv_{\ell_2}\bigl(
 \{\pm e_i/b_i:i\ge1\}
 \cup\{\pm w:w\in W_A,\ A\in\cA_n,\ n\ge1\}\bigr),
 \qquad 1/\infty=0.
\]
For $\lambda\in K^\circ$ put $x_i=|\lambda_i|/2$. We have
$0\le x_i\le b_i$ by the vectors $e_i/b_i$ in $K$. If $I_A\ne\emptyset$,
\eqref{eq:hull-net} gives $\lambda_{I_A}/2\in K_A$, so
$x_{I_A}\in K_A$ by unconditionality.
For $i\in I_A$, we have $V_i(x_i)<\infty$ since $x_{I_A}\in K_A$.
Hence $\int_1^{x_i\vee1}\sigma_i(u)\,du\le V_i(x_i)$. Thus
\eqref{eq:upper-a-random-bound} and \eqref{eq:upper-HA-bound} give
\[
 \begin{aligned}
 \sum_i |a_i(t)|x_i
 &\le CB+\sum_{n\ge1}\sum_{i\in I_{\cA_n(t)}}
                 \int_1^{x_i\vee1}H_{\cA_n(t),i}(u)\,du\\
 &\le CB+2\sum_{n\ge1}r_{n-1}(t)
                 \sum_{i\in I_{\cA_n(t)}}V_i(x_i)\\
 &\le CB+2\sum_{n\ge1}2^n r_{n-1}(t)\le CB .
 \end{aligned}
\]
Thus $|\langle a(t),\lambda\rangle|\le CB$ for every $\lambda\in K^\circ$
and $t\in T$. Since $K$ is closed, convex and symmetric, the bipolar
theorem gives $a(T)\subset CB K$.

For $A\in\cA_n$, $n\ge1$, put
\begin{equation}\label{eq:hull-gA}
 c_A=2^nr_A,\qquad g_A=h_A/c_A\quad(c_A>0).
\end{equation}
Take $g_A=0$ when $c_A=0$. By \eqref{eq:upper-hA-size} and
Lemma~\ref{lem:upper-moment},
\begin{equation}\label{eq:hull-gA-moment}
 \|Y_{g_A}\|_p\le C\bigl(1+2^{-n}p\bigr),\qquad p\ge1.
\end{equation}
Moreover,
\[
 v(t)=\sum_{n\ge1}c_{\cA_n(t)}g_{\cA_n(t)},\qquad
 \sum_{n\ge1}c_{\cA_n(t)}\le2B.
\]
Thus $v(T)\subset2B\clconv_{\ell_2}\{\pm g_A:A\in\cA_n,\ n\ge1\}$.

For $n\ge1$ put
$U_n=\{g_A:A\in\cA_n\}\cup\bigcup_{A\in\cA_n}W_A$.
Choose a sequence $(u_j)_{j\ge1}$ by enumerating the finite sets
$U_1,U_2,\ldots$ in this order (allowing repetitions).
If $u_j$ is taken from $U_n$, then
\[
 j\le\sum_{k=1}^n 2^{2^k}\bigl(1+5^{2^{k-1}}\bigr)
 \le \exp(C2^n),\qquad \ell_{2j-1}\le C2^n.
\]
Take
\[
 s_{2j-1}=CBu_j,\qquad s_{2j}=CB e_j/b_j.
\]
Since $T-t_0\subset a(T)+v(T)$, for $C$ large enough
\[
 T-T\subset\clconv_{\ell_2}\{\pm s_j:j\ge1\}.
\]
By \eqref{eq:hull-w-moment}, \eqref{eq:hull-gA-moment}, and
$\|Y_{e_j/b_j}\|_\infty\le1$, we have
$\|Y_{s_j}\|_{\theta\ell_j}\le C_\theta B$.
\end{proof}

\begin{proof}[Proof of Corollary~\ref{cor:intro-hull}]
For bounded $T\subset\R^n$, $S_Y(T)<\infty$. Take $(s_j)$ as in
\eqref{eq:countable-hull} and put $\widehat s_j=(s_j)_{\{1,\ldots,n\}}$.
Then
\[
 T-T\subset\clconv\{\pm\widehat s_j:j\ge1\},\qquad
 \|Y_{\widehat s_j}\|_p\le\|Y_{s_j}\|_p\quad(p\ge1).
\]
The second inequality follows from Jensen's inequality and the
independence of $Y_{\widehat s_j}$ and $Y_{s_j-\widehat s_j}$.
Put $K_0=\conv\{\pm\widehat s_j:j\ge1\}$. Choose $m$ and $\delta>0$
such that
\[
 \operatorname{span}K_0=\operatorname{span}\{\widehat s_1,\ldots,\widehat s_m\},
 \qquad
 \{x\in\operatorname{span}K_0:\|x\|_2<\delta\}
 \subset\conv\{\pm\widehat s_j:1\le j\le m\}\subset K_0.
\]
For $x\in\overline K_0$, take $y\in K_0$ with $\|x-y\|_2<\delta$.
Then
\[
 x=y+(x-y)\in K_0+K_0=2K_0.
\]
Use the vectors $2\widehat s_j$ in \eqref{eq:intro-hull}.
\end{proof}

\subsection{A converse estimate}

\begin{proposition}\label{prop:hull-converse}
Let $s_j\in\ell_2$, $j\ge1$, satisfy
\begin{equation}\label{eq:hull-converse-assumption}
 \|Y_{s_j}\|_{\ell_j}\le M,
 \qquad j\ge1.
\end{equation}
Then
\begin{equation}\label{eq:hull-converse-max}
 \E\sup_{j\ge1}|Y_{s_j}|\le CM.
\end{equation}
Consequently, if
\[
 T-T\subset\clconv\{\pm s_j:j\ge1\},
\]
then $S_Y(T)\le CM$.
\end{proposition}

\begin{proof}
Put $W=\sup_j|Y_{s_j}|$. For $M>0$ and $u\ge M$,
Markov's inequality gives
\[
 \P(|Y_{s_j}|\ge u)\le(M/u)^{\ell_j}.
\]
Hence
\begin{equation}\label{eq:hull-tail}
 \E\bigl[|Y_{s_j}|\ind_{\{|Y_{s_j}|\ge e^2M\}}\bigr]
 \le M e^{2(1-\ell_j)}\left(1+\frac1{\ell_j-1}\right).
\end{equation}
Since $\ell_j=\log(e+j)$, we obtain
\[
 \E W\le e^2 M+\sum_{j\ge1}
 \E\bigl[|Y_{s_j}|\ind_{\{|Y_{s_j}|\ge e^2 M\}}\bigr]\le CM.
\]
This holds also for $M=0$. The set
\[
 \{h\in\ell_2:\E(|Y_h|-W)_+=0\}
\]
is closed by \eqref{eq:intro-process} and contains
$\conv\{\pm s_j:j\ge1\}$. Thus for $t_0\in T$ and finite
$F\subset T$ we have $\max_{t\in F}|Y_{t-t_0}|\le W$ almost surely.
This gives $S_Y(T)\le\E W$.
\end{proof}

\medskip
\noindent\textbf{Acknowledgment}. The work of Hanchao Wang was supported by the
National Key R\&D Program of China (No. 2024YFA1013501), the National Natural
Science Foundation of China (No. 12571162), and Shandong Provincial Natural
Science Foundation (No. ZR2024MA082).

\medskip
\noindent\textbf{AI use}. AI was used for language editing.
All mathematical ideas, results and proofs are due to the authors.

\medskip
\noindent\textbf{Author contributions}. All authors contributed equally to this work.

\printbibliography

\end{document}